\let\ORIlabel\label
\let\ORIrefstepcounter\refstepcounter
\AddToHook{package/hyperref/before}{
   \let\label\ORIlabel 
   \let\refstepcounter\ORIrefstepcounter}
\documentclass[preprint]{siamart220329}

\usepackage{lipsum}
\usepackage{amsfonts}
\usepackage{graphicx}
\usepackage{mathtools}
\usepackage{algorithmic}
\usepackage{subfigure}
\usepackage{bm}
\usepackage{amssymb,amscd,amsxtra,calc}
\usepackage{cmmib57}
\usepackage{multirow}
\usepackage{mathrsfs}
\usepackage{graphicx}
\usepackage{bm}
\usepackage{hyperref}
\usepackage{amsmath}
\usepackage{amssymb}
\usepackage{ntheorem}
\usepackage{cleveref}
\usepackage{graphicx}
\usepackage{accents}
\usepackage{tikz} 
\usepackage{tikz-cd}
\usepackage{subfigure}
\usepackage{color}
\usepackage{ragged2e,array}
\usepackage{stmaryrd}
\usepackage{comment}
\usepackage{multirow}
\usepackage{tabularray}
\usepackage[font=footnotesize,labelfont=bf]{caption} 
\usepackage[section]{placeins} 
\allowdisplaybreaks[3]
\usepackage{import} 
\usepackage{colonequals}
\usepackage{tabularx}
\usepackage{cite}
\usepackage{todonotes}
\usepackage{subcaption}

\newcommand{\pef}[1]{\todo[inline, color=blue!30]{\textbf{PEF: }#1}}

\usepackage{multirow}
\newsiamthm{remark}{Remark}
\newsiamremark{hypothesis}{Hypothesis}
\crefname{hypothesis}{Hypothesis}{Hypotheses}
\newsiamthm{claim}{Claim}

\newsiamremark{assumption}{Assumption}
\newsiamthm{example}{Example}
\newsiamthm{problem}{Problem}

\renewcommand{\d}{\mathrm{d}}

\renewcommand{\l}{\left}
\renewcommand{\r}{\right}

\newcommand\bA{{\boldsymbol{A}}}
\newcommand\bE{{\boldsymbol{E}}}
\newcommand\bB{{\boldsymbol{B}}}

\newcommand\bC{{\boldsymbol{C}}}
\newcommand\bH{{\boldsymbol{H}}}

\newcommand{\cE}{\mathcal{E}}
\newcommand{\cH}{\mathcal{H}}
\renewcommand{\div}{\mathrm{div}}
\newcommand{\curl}{\mathrm{curl}}

\newcommand{\includegraphicswithlegend}[5][1.0]{%
	\setlength{\unitlength}{#2}%
  \begin{picture}(1,1)%
    \put(0,0){\includegraphics[width=#1\unitlength]{#3}}%
    \put(0.95,0.27){\includegraphics[width=.1\unitlength,height=.65\unitlength]{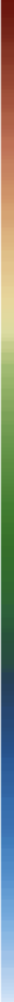}}%
    \put(0.95,0.95){\scriptsize{#4}}%
    \put(0.95,0.17){\scriptsize{#5}}%
  \end{picture}%
}

\headers{Global and local helicity-preservation}{P.~E.~Farrell, M.~He, K.~Hu and G. Zhang}

\title{Global and local helicity-preservation in the finite element discretization of magnetic relaxation \thanks{Submitted to the editors DATE.
\funding{This work was funded by the Engineering and Physical Sciences Research Council [grant number EP/W026163/1],
the Science and Technology Facilities Council [grant number UKRI/ST/B000495/1],
the Donatio Universitatis Carolinae Chair ``Mathematical modelling of multicomponent systems",
the UKRI Digital Research Infrastructure Programme through the Science and Technology Facilities Council's Computational Science Centre for Research Communities (CoSeC),
the Swedish Research Council under grant no.~Z2021-06594 while in residence at Institut Mittag-Leffler in Djursholm, Sweden,
the European Research Council (ERC Starting Grant, project 101164551 GeoFEM), and by
a Royal Society University Research Fellowship (URF$\backslash$R1$\backslash$221398).
For the purpose of open access, the authors have applied a CC BY public copyright licence to any author accepted manuscript arising from this submission. No new data were generated or analysed during this work. }}}

\author{
Patrick E.~Farrell\thanks{Mathematical Institute, University of Oxford, UK
     and
     Mathematical Institute, Faculty of Mathematics and Physics, Charles University, Czechia ({patrick.farrell@maths.ox.ac.uk})}
	\and
	Mingdong He\thanks{Mathematical Institute, University of Oxford, UK ({mingdong.he@maths.ox.ac.uk})}		
	\and Kaibo Hu\thanks{Mathematical Institute, University of Oxford, UK
  ({kaibo.hu@maths.ox.ac.uk})}
	\and Ganghui Zhang\thanks{Mathematical Institute, University of Oxford, UK
  ({ganghui.zhang@maths.ox.ac.uk})}
}

\usepackage{amsopn}

\usepackage{mathrsfs}
\usepackage{amsfonts}
\usepackage{stmaryrd}
\usepackage{amsmath}
\usepackage{amssymb}
\usepackage{multirow}
\usepackage{tabularx}
\usepackage{booktabs}
\usepackage{cite}
\usepackage{verbatim}
\usepackage{color}
\usepackage{graphicx}
\usepackage{subfigure}
\usepackage{epsfig,epstopdf,color}
\allowdisplaybreaks[4]

\begin{document}
\maketitle
\begin{abstract}
Magnetic relaxation drives plasma toward lower-energy equilibria under helicity constraints. In ideal magnetohydrodynamics (MHD), helicity is locally conserved, while resistive theories such as Taylor relaxation preserve only global helicity. This distinction has important implications for structure-preserving numerical methods.
We compare three finite element formulations: an unconstrained scheme that does not conserve helicity, a mixed method based on finite element exterior calculus that preserves discrete local helicity on magnetically closed subdomains, and a Lagrange multiplier approach that enforces only global helicity conservation. Numerical experiments with magnetic knots and braids show that helicity-based constraints provide effective topological barriers when the relevant helicity-type invariant is nonzero, but do not fully characterize braided field-line topology when it vanishes. These results clarify both the strengths and the possible limitations of helicity-based
structure-preserving finite element methods for magnetic relaxation.
\end{abstract}

\begin{keywords}
magnetohydrodynamics,
structure-preservation,
Lagrange multiplier,
finite element exterior calculus,
magnetic helicity,
magnetic relaxation.
\end{keywords}

\begin{MSCcodes}
 65N30, 65L60, 76W05
\end{MSCcodes}

\section{Introduction}
Magnetic relaxation describes the process by which a magnetized plasma reorganizes its magnetic field toward a lower-energy equilibrium. Magnetic relaxation is a fundamental process in plasma physics, playing a central role in the understanding of magnetic equilibria in both natural and laboratory plasmas. In magnetically ideal situations, this reorganization is described by the ideal magnetohydrodynamic (MHD) equations, or their simplification, the magneto-frictional (MF) equations \cite{he2025helicity,taylor1974relaxation,chodura3DCodeMHD1981,yeatesLimitationsMagnetofrictionalRelaxation2022} posed on a bounded, contractible, Lipschitz domain $\Omega \subset \mathbb{R}^3$:
\begin{subequations}\label{eqn:magneto-frictional}
\begin{align}
\partial_t \bm{B} + \nabla\times \bm{E}&=\bm{0},\label{eqn:magnetic-advection}\\
   \bm{E}+ \bm{u}\times\bm{B} &= \bm{0},\label{eqn:mf-Ohms} \\
   \bm{u} &= \tau \bm{j}\times\bm{B},\label{eqn:magneto-frictional3}\\
    \bm{j}& = \nabla\times \bm{B},
\end{align}
\end{subequations}
where $\bm{B}$ is the magnetic field, $\bm{E}$ is the electric field, $\bm{u}$ is the velocity, $\bm{j}$ is the current, and $\tau > 0$ is a coupling parameter.
These MF equations possess the same equilibria as the ideal MHD equations, allowing their study with less computational expense.

Both the ideal MHD and MF systems impose topological constraints of the magnetic fields.
Helicity is a quantitative measure of the knottedness and linking of the magnetic fields, which is defined on a domain $\Omega_s \subset \Omega$ as
\begin{equation}
\cH(\Omega_s) \coloneqq \int_{\Omega_s} \bm A\cdot \bm B\,\d x,
\end{equation}
where $\bm A$ is any magnetic potential satisfying $\nabla\times \bm A=\bm B$. The quantity $\cH(\Omega)$, known as Woltjer's invariant \cite{woltjer1958theorem}, describes the total averaged knotting of the magnetic field in the domain. Hereafter, we refer to this quantity as the {\it global} helicity since it is integrated over the entire domain. Woltjer's invariant is conserved in ideal MHD. In fact, $\cH(\Omega_s)$ is conserved for any {\it magnetically closed domain} $\Omega_{s}\subset \Omega$, i.e., any subdomain $\Omega_{s}$ such that $\bm B$ is tangent to its boundary ($\bm B|_{\partial \Omega_{s}}\cdot \bm n_{\partial \Omega_{s}}=0$). A more geometric description is that the volume form $\bm A\wedge \bm B$ is transported by the flow and is thus invariant. Hereafter, we refer to the conservation of $\cH(\Omega_s)$ for $\Omega_s \subsetneq \Omega$ as the {\it local} conservation of helicity. In both ideal MHD and the MF equations, the local helicity is conserved in time for any suitable $\Omega_s$.
For a comprehensive review of magnetic relaxation and its topological constraints, we refer to \cite{yeates2019magnetohydrodynamic}.

Numerical simulations of magnetic relaxation yield nonphysical results if the
numerical schemes do not appropriately preserve helicity~\cite{he2025helicity};
if helicity is not conserved, then the magnetic field relaxes without topological
constraints towards the trivial zero state $\bm B=\bm{0}$. In our previous
work~\cite{he2025helicity} we introduced a structure-preserving scheme that
conserves discrete local helicity through the introduction of auxiliary variables
(fields over the domain). In this work we investigate whether such local helicity preservation is necessary
to predict the relaxed state, or whether preserving only the global helicity
\(\mathcal H(\Omega)\), through a single real-valued Lagrange multiplier rather
than auxiliary fields, already suffices. Realizing the
distinction between local and global helicities, we reframe the central question as
follows: we ask what can, and what cannot, be controlled by enforcing helicity constraints at the discrete level. In particular, we compare schemes preserving no helicity,
only global helicity, and discrete local helicity, and examine how these choices
affect the relaxed state, the relaxation pathway, and the retention of magnetic
structure.

 We therefore compare three finite element discretizations. The first takes no special care to conserve helicity on discretization and leads to nonphysical trivial states. The second is our projection-based scheme from \cite{he2025helicity}, which introduces the projection of the magnetic field onto another function space so as to conserve helicity for every magnetically closed domain. The third is a novel scheme where the global helicity constraint is enforced with a Lagrange multiplier. This corresponds to Taylor's relaxation theory \cite{moffatt2015magnetic,taylor1974relaxation}, which assumes that the global
helicity is approximately conserved in the real plasma relaxation, although local helicities can change due to reconnections.
 The first leads to $\bm B=\bm{0}$ and thus with vanishing current $\bm j = \nabla \times \bm B = \bm 0$; the second leads to so-called `nonlinear force-free fields' with $\bm j = \alpha(\bm x)\bm B$; the third leads to so-called `linear force-free fields' with $\bm j = \alpha \bm B$ for a constant $\alpha$. In the second and third cases, equilibrium is described by the vanishing of the Lorentz force
\begin{equation}\label{eqn:zero-Lorentz}
\bm j \times \bm B = \bm 0.
\end{equation}

The Lagrange multiplier scheme enforces the same global invariant as Taylor-type
relaxation models, but it does not constrain the redistribution of helicity among
magnetically closed subregions. The projection-based scheme imposes stronger
local helicity constraints. Our numerical experiments show that the distinction
between these two levels of constraint is important for helicity-carrying fields,
but is not sufficient by itself to guarantee preservation of the full topology of
zero-helicity braids. 
interpreted as a numerical analogue of local reconnection, but whether this is
physical or spurious depends on the modelling regime under consideration.



  However, real physical situations are not ideal, breaking helicity conservation. Taylor's relaxation theory assumes that the global helicity is approximately conserved to high accuracy, although local helicities are not \cite{taylor1974relaxation,yeates2019magnetohydrodynamic,moffatt2015magnetic}. A consequence of the Taylor relaxation theory is that turbulent plasmas relax toward a linear force-free state under only the global constraint.
  In the latter part of this paper, we discuss the possibility of using the local reconnections arising from numerical errors in the Lagrange multiplier approach as an approach for simulating Taylor relaxation. In short, {\it numerical errors (reconnections) in local helicity might reflect the reconnection of magnetic fields in the real physical problem, leading to physically relevant solutions}.  The physics of magnetic relaxation can be a decisive factor for the choice of numerical schemes, especially at the level of helicity preservation.

Another important consideration in magnetic relaxation is its performance on a wider range of topological configurations. While nonzero helicity implies nontrivial topology (e.g.~linked tubes or rings), the converse is not true: there can be topologically nontrivial magnetic fields with zero helicity. The analysis in \cite{he2025helicity} only applies to those fields with nonzero helicity, by proving a discrete Arnold inequality that guarantees a lower bound on the evolution of the magnetic energy. 
For topologically nontrivial fields with zero helicity, such as magnetic braid
configurations, Arnold-type helicity barriers do not provide a positive lower
bound on the magnetic energy. These examples therefore test the limitations of
helicity-based structure preservation: even preserving discrete local helicity
does not necessarily preserve the full braided field-line topology.

For open magnetic configurations such as braids, where magnetic flux crosses part
of the boundary, the classical helicity is not directly gauge invariant. This has
motivated other notions of helicity, like relative
helicity \cite{berger1984topological,finn1985magnetic} and Bevir--Gray helicity \cite{bevir1980relaxation}. These notions provide
important theoretical diagnostics for open magnetic fields, but their direct use
as structure-preserving finite element invariants remains largely open. In this
work, our numerical experiments show that the generalized helicity introduced in
our previous work \cite{he2025helicity} is not merely an analytical quantity, but also a practical
computable helicity for braided magnetic fields. In particular, it distinguishes
zero-generalized-helicity braids from helicity-carrying braids and provides an
effective topological constraint when nonzero.

More generally, the past decades have seen significant progress in finite element methods for MHD systems. In particular, schemes based on the finite element exterior calculus (FEEC) \cite{arnoldFiniteElementExterior2006,arnoldFiniteElementExterior2010,ArnoldFiniteElementExterior2018} have been developed that precisely preserve important structure, such as the magnetic Gauss law and helicity conservation \cite{hu2017stable,huHelicityconservativeFiniteElement2021,gawlikFiniteElementMethod2022,LaakmannStructurepreservinghelicityconservingfinite2023,maoIncompressibilityDivB0Preserving2025,zhangMassKineticEnergy2022,zhangMEEVCDiscretizationTwodimensional2024, BlickhanMRXdifferentiable3D2025,da2025error,ma2016robust}. Extensive numerical results demonstrate that standard finite element methods that do not explicitly enforce helicity conservation produce qualitatively wrong solutions, as discretization errors destroy topological structures, whereas helicity-preserving methods evolve toward physically meaningful solutions. While Lagrangian discretizations have been widely employed to track these constraints \cite{craig1986dynamic,longbottom1998magnetic, craig2005parker, wilmot2009magnetic, wilmot2009magneticparallel,craig2014current,candelaresiMimeticMethodsLagrangian2014a,zhou2014variational,zhou2016formation, zhou2017constructing}, Eulerian discretizations have advantages in stability and the handling of complex geometries \cite{he2025helicity,BlickhanMRXdifferentiable3D2025}. 

   The remainder of this paper is organized as follows.
In \Cref{sec:Preliminaries}, we introduce preliminaries and the magnetic topologies we consider, including magnetic knots and magnetic braids. In \Cref{sec:Non-con}, we propose a non-conservative scheme based on a na\"ive formulation. The projection-based finite element method of \cite{he2025helicity} is reviewed and discussed in \Cref{sec:MixedFEM}. Then in \Cref{sec:LM}, we propose a global structure-preserving scheme via Lagrange multipliers. In \Cref{sec:Numerical-experiment}, we present numerical results, and compare the non-conservative scheme, the projection-based method, and the Lagrange multiplier method, to explore the significance of global and local helicity preservation. In \Cref{sec:discussion}, we further discuss the background physical meaning of the two structure-preserving schemes. Finally we draw some conclusions in \Cref{sec:conclusion}.

\section{Preliminaries: helicity and magnetic topology}\label{sec:Preliminaries}
Let $\Omega$ be a bounded Lipschitz domain in $\mathbb{R}^3$;
if not otherwise specified, we assume that $\Omega$ is contractible. 
Let $\bm n$ denote the outward-pointing unit normal vector on $\partial\Omega$.
We use $\|\cdot\|$ and $(\cdot, \cdot)$ to denote the $L^2(\Omega)$ norm and inner product respectively, allowing $L^2(\Omega)$ to denote both the scalar- and vector-valued spaces. The Hilbert spaces $H^1$, $ \bm H(\curl) $ and $\bm H(\div) $ are defined as in e.g.~\cite{arnoldFiniteElementExterior2010}. We further introduce subspaces $H^1_0$, $\bm H_0(\curl)$ and  $\bm H_0(\div) $ with homogeneous boundary conditions on $\partial\Omega$.

The 3D de~Rham complex with homogeneous boundary conditions reads:
\begin{subequations}
\begin{equation}
    \begin{tikzcd}
    0\arrow{r}&H_0^1 \arrow[r, "\nabla"] &
   \bm H_0(\curl)\arrow[r,"\nabla\times"] &
   \bm H_0(\div) \arrow[r,"\nabla\cdot"] &
    L_0^2 \arrow{r} & 0.
    \end{tikzcd}
    \label{eq:deRhamcomplex}
\end{equation}
This complex \eqref{eq:deRhamcomplex} is exact on contractible domains.
We will use finite-element subcomplexes of \eqref{eq:deRhamcomplex} for discretization;
families of such subcomplexes are well-known, consisting of N\'ed\'elec \cite{nedelec1-0}, Raviart--Thomas\cite{raviart2006mixed}, and Brezzi--Douglas--Marini elements \cite{brezzi1985two}, each extending to arbitrary spatial dimensions and polynomial degrees.
Adopting the notation of \cite{ArnoldFiniteElementExterior2018}, we denote such a subcomplex by
\begin{equation}
    \begin{tikzcd}
0\arrow{r}& H_{0}^{1,h}\arrow[r, "\nabla"] &
   \bm H_0^h(\curl)\arrow[r,"\nabla\times"] &
   \bm H_0^h(\div)\arrow[r,"\nabla\cdot"] &
    L_0^{2,h} \arrow{r} & 0.
    \end{tikzcd}
    \label{eq:subcomplex}
\end{equation}
\end{subequations}
We require that \eqref{eq:subcomplex} is exact on contractible domains.

Define the magnetic energy
\begin{equation}\label{eqn:energy}
    \cE=\int_{\Omega} \bm B\cdot \bm B\ \d x.
\end{equation}
We always assume the whole domain $\Omega$ is a magnetic closed domain, and denote the global helicity $\cH:=\cH(\Omega)$. The Arnold inequality \cite{arnold1974asymptotic} is a crucial result imposing a topological barrier on the energy achievable by magnetic relaxation. It states that
\begin{equation}
    |\cH| \le C\cE,
\end{equation}
for a constant $C>0$. We close \eqref{eqn:magneto-frictional} with
the boundary conditions
\begin{equation} \label{eqn:bcs}
     \bm{B}\cdot\bm{n} = 0,  \quad \bm j \times \bm n = \bm{0}, \quad \bm{u}\cdot\bm{n} = 0,\quad \text{on } \partial \Omega,
\end{equation}
which ensure that the magnetic energy $\cE$ decreases until reaching its equilibrium. The local helicity on any magnetic subdomain, and hence the global helicity, is conserved.

 Magnetic fields can exhibit non-trivial topology, arising from the winding, linking, or tangling of magnetic field lines. A classic example is the magnetic knot, in which field lines form closed loops that are linked or knotted in a topologically non-trivial way.
 In contrast, magnetic braids consist of open flux tubes whose field lines are tangled between two boundaries (e.g., photospheric footpoints) but do not necessarily form closed, linked loops \cite{YeatesWilmotSmithHornig2010}.
Because opposite twists can cancel, magnetic braids often possess zero net helicity. Consequently, the global helicity $\mathcal{H}$ is insensitive to their internal topological complexity and fails to distinguish between different braiding patterns.
Another example is the Borromean rings configuration, where the global helicity vanishes despite non-trivial triple linking.

\begin{figure}[h!]
    \centering
    \begin{tabular}{ccc}
        \hspace{-1.18cm}{\includegraphics[width=0.34\textwidth]{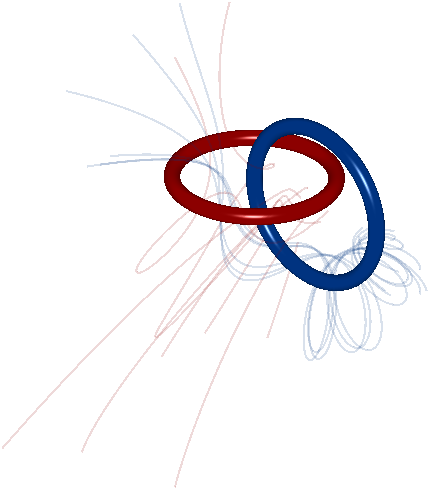}} & 
        \hspace{0.6cm}{\includegraphics[width=0.28\textwidth]{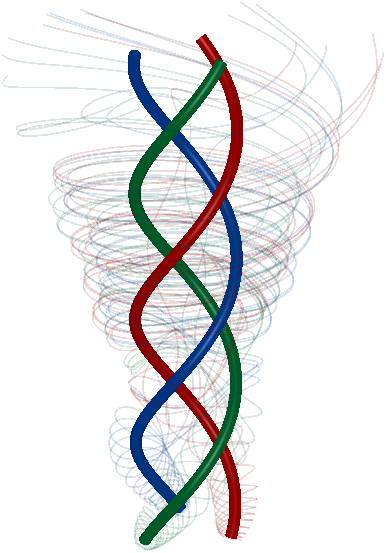}} & 
        \hspace{-0.28cm}{\raisebox{0.7cm}{\includegraphics[width=0.35\textwidth]{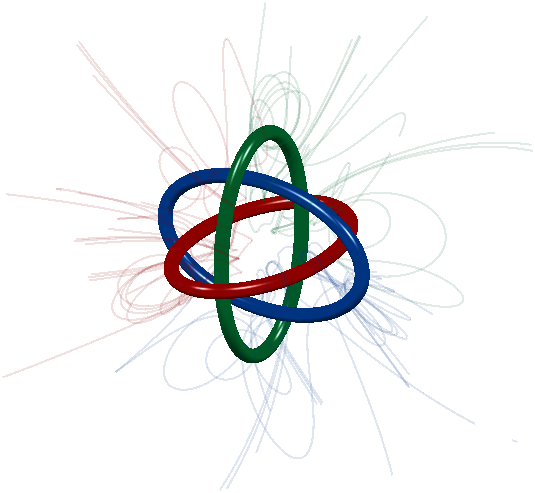}}} \\
        (a) Magnetic knot &\ \ \ \ (b) Magnetic braid & (c) Borromean ring 
    \end{tabular}
    \caption{Illustration of non-trivial topology of magnetic fields. (a) Magnetic knot with non-zero helicity, (b)  Magnetic braid with zero helicity and (c) Borromean ring with zero helicity. Here the thin lines represent the magnetic field lines surrounding the high-intensity core tubes.}
\end{figure}


This limitation of helicity in describing richer topological configurations inspires the investigation of other measures, such as higher-order linking invariants \cite{ArnoldTopologicalMethodsHydrodynamics2021,massey1998higher} for configurations like the Borromean rings, and more refined, local, or field-line-based definitions, such as field-line mapping, topological entropy, or distributions of field-line helicity, to describe braided topology \cite{YeatesHornig2013,YeatesWilmotSmithHornig2010}.
 These quantities reveal local winding and stretching even when global helicity is zero. However, reflecting these structures in finite element computation is still largely open, and is beyond the scope of this paper.

The fundamental structural differences between magnetic knots (closed, helicity-carrying) and magnetic braids (open, often helicity-neutral yet topologically rich) 
provide test cases not only for the effectiveness, but also for the
limitations, of helicity-preserving numerical schemes.
We will investigate the dynamics of these topological configurations with algorithms for magnetic relaxation that enforce varying degrees of helicity conservation during energy minimization. 

\section{Non-conservative scheme}\label{sec:Non-con}
A natural discretization of the MF equations \eqref{eqn:magneto-frictional}, \eqref{eqn:bcs} using variables from a de~Rham complex follows from \cite{hu2017stable,hu2019structure}.

\smallskip

\begin{problem}[Non-conservative scheme]
   At time step $n\ge 0$,  find $$(\bm B_h^{n+1}, \bm E_h^{n+1/2},\bm j_h^{n+1/2})\in \bm H_0^h(\div)\times[\bm H^h_0(\curl)]^2$$ such that for any test function $(\bm C_h,\bm F_h, \bm k_h) \in \bm H^h_0(\div)\times[\bm H^h_0(\curl)]^2$,
\begin{subequations}\label{eq:nonconservative-scheme}
    \begin{align}
        \left(\frac{\bm B_h^{n+1} - \bm B_h^{n}}{\Delta t}, \bm{C}_h\right) + (\nabla\times\bm{E}_h^{n+1/2}, \bm{C}_h) &= 0,  \\
        (\bm{E}_h^{n+1/2}, \bm{F}_h) +  \tau((\bm j_h^{n+1/2}\times \bm B_h^{n+1/2}) \times \bm{B}^{n+1/2}_h, \bm{F}_h) &= 0,  \\
        (\bm{j}_h^{n+1/2}, \bm{k}_h) -(\bm{B}_h^{n+1/2}, \nabla\times\bm{k}_h)&=0 ,
    \end{align}
    \end{subequations}
\end{problem}
where we use the Crank--Nicolson temporal discretization method for $\bm B_h$, that is, $\bm B_h^{n+1/2}=(\bm B^{n+1}_h+\bm B^{n}_h)/2$, and consider $\bm{E}_h^{n+1/2}$ and $\bm j_h^{n+1/2}$ as independent variables.
The scheme preserves the magnetic Gauss law $\nabla\cdot\bm B=0$ and the energy decay.

\smallskip

\begin{theorem}\label{thm:Hdiv}
    Assume that the initial condition satisfies $\nabla\cdot \bm B_h^0=0$ and that \\$(\bm B_h^{n+1}, \bm E_h^{n+1/2}, \bm j_h^{n+1/2})$ is a solution of \eqref{eq:nonconservative-scheme}. Then the energy is decreasing and the discrete Gauss law holds, that is,
        \begin{equation}\label{energy-law}
            \cE^{n+1}_h\le \cE^n_h\le \ldots \le \cE^0_h,
        \end{equation}
        and 
        $$
        \nabla\cdot \bm B^{n+1}_h=0,\qquad n\ge 0,
        $$
 where $\cE^n_h=\l( \bm B_h^{n},\bm B_h^{n}\r)$.
\end{theorem}

\smallskip
However, the discrete helicity is not conserved at either the local or global level due to numerical pollution \cite{huHelicityconservativeFiniteElement2021}. Moreover, the lack of the discrete Arnold inequality leads to nonphysical solutions, as shown in \Cref{sec:Numerical-experiment} below.

\section{Projection-based mixed finite element scheme}\label{sec:MixedFEM}
We briefly review the model and schemes presented in \cite{he2025helicity}. The projection-based mixed finite element scheme ({\it projection-based scheme}, in short) in \cite{he2025helicity} introduces an additional auxiliary variable, and  preserves discrete local helicity on magnetically closed subdomains.

\smallskip

\begin{problem}[projection-based scheme]
   At time step $n\ge 0$,   find $$(\bm B_h^{n+1}, \bm E_h^{n+1/2}, \bm j_h^{n+1/2}, \bm H_h^{n+1/2})\in \bm H_0^h(\div) \times [\bm H_0^h(\curl)]^3$$ such that for all test functions $(\bm C_h, \bm F_h, \bm k_h,\bm D_h)\in \bm H_0^h(\div) \times [\bm H_0^h(\curl)]^3$,
    \begin{subequations}\label{eq:mixed-scheme}
    \begin{align}
        \left(\frac{\bm B_h^{n+1} - \bm B_h^{n}}{\Delta t}, \bm{C}_h\right) + (\nabla\times\bm{E}_h^{n+1/2}, \bm{C}_h) &= 0,  \\
        (\bm{E}_h^{n+1/2}, \bm{F}_h) +  \tau((\bm j_h^{n+1/2}\times \bm H_h^{n+1/2})\times \bm{H}^{n+1/2}_h, \bm{F}_h) &= 0,  \label{eq:mixed-scheme2}\\
        (\bm{j}_h^{n+1/2}, \bm{k}_h)-(\bm{B}_h^{n+1/2}, \nabla\times\bm{k}_h) &=0 ,  \\
        (\bm{H}_h^{n+1/2}, \bm{D}_h) -(\bm{B}_h^{n+1/2}, \bm{D}_h)&= 0.\label{eq:mixed-scheme5}
    \end{align}
    \end{subequations}
\end{problem}
The time discretization strategy is the same as that of \Cref{sec:Non-con}; the difference is the introduction of the auxiliary variable, where $\bm H \in \bm H_0^h(\curl)$ is a projection of the magnetic field $\bm B \in \bm H_0^h(\div)$ to avoid helicity pollution in discretization. This auxiliary variable is exactly that indicated by the framework of \cite{andrews2024enforcing}. For the solver, we apply Newton iteration and solve the linearized system by the direct solver MUMPS~\cite{amestoy2001}. 

\smallskip

\begin{theorem}\label{thm:Mix}
    Assume that the initial condition satisfies $\nabla\cdot \bm B_h^0=0$ and that\\ $(\bm B_h^{n+1}, \bm E_h^{n+1/2}, \bm j_h^{n+1/2}, \bm H_h^{n+1/2})$ is a solution of \eqref{eq:mixed-scheme}. Then the energy is nonincreasing, the discrete Gauss law holds, and the (global) discrete Arnold inequality holds 
    \begin{equation}
             |\mathcal{H}^{n}_h|\leq C\mathcal{E}^{n}_h,\qquad n\ge 0,
        \end{equation} 
        where $\mathcal{H}^{n}_h=(\bA^{n}_h,\bB^{n}_h)$. Moreover, for magnetically closed subdomain $\Omega_{s,h}$ such that $\bm B^{n+1/2}_{h}|_{\Omega_{s,h}}\in \bm H^h_0(\div,\Omega_{s,h})$ and $\bm H^{n+1/2}_{h}|_{\Omega_{s,h}},\bm E^{n+1/2}_{h}|_{\Omega_{s,h}}\in \bm H^h_0(\curl,\Omega_{s,h})$, the local helicity is conserved in the sense
    \begin{equation}\label{local-helicity-conservation}
    \int_{\Omega_{s,h}}\bm B_h^{n+1}\cdot \bm A_h^{n+1}\, \d x =\int_{\Omega_{s,h}}\bm B_h^{n}\cdot \bm A_h^{n}\, \d x.
\end{equation}
\end{theorem}
\begin{proof}
    The proof of the energy decay, the discrete Gauss law and the discrete Arnold inequality can be found in \cite{he2025helicity}. For the preservation of local helicity, we first notice that by definition there exists a vector potential 
    \begin{equation}
        \bm A^{n+{1/2}}_h=\frac{\bm A^{n+{1}}_h+\bm A^{n
    }_h}{2}\in \bm H^h_0(\curl,\Omega_{s,h}),\qquad \nabla \times \bm A^{n+{1/2}}_h=\bm B^{n+{1/2}}_h.
    \end{equation}
Therefore, using integration by parts at this subdomain, we get (see \cite[Theorem 5]{huHelicityconservativeFiniteElement2021} for a similar argument)
    \begin{align*}
        &\int_{\Omega_{s,h}}\bm B_h^{n+1}\cdot \bm A_h^{n+1}\, \d x -\int_{\Omega_{s,h}}\bm B_h^{n}\cdot \bm A_h^{n}\, \d x\\
        &=\int_{\Omega_{s,h}}\l(\bm B_h^{n+1}-\bm B_h^{n} \r)\cdot \bm A_h^{n+1/2}\, \d x +\int_{\Omega_{s,h}}\l(\bm A_h^{n+1}-\bm A_h^{n} \r)\cdot \bm B_h^{n+1/2}\, \d x\\
        &=2\int_{\Omega_{s,h}}\l(\bm B_h^{n+1}-\bm B_h^{n} \r)\cdot \bm A_h^{n+1/2}\, \d x.
    \end{align*}
 By assumption, we can take zero extension $\widetilde{\bE}_{h}^{n+1/2}=\bE_h^{n+1/2}\mathbf{1}_{\Omega_{s,h}}\in \bm H_0^h(\curl)$,  $\widetilde{\bA}_{h}^{n+1/2}=\bA_h^{n+1/2}\mathbf{1}_{\Omega_{s,h}}\in \bm H_0^h(\curl)$, $\widetilde{\bH}_{h}^{n+1/2}=\bH_h^{n+1/2}\mathbf{1}_{\Omega_{s,h}}\in \bm H_0^h(\curl)$, and take the test $\bC_h=\mathbb{Q}_h^{\div}\widetilde{\bA}_{h}^{n+1/2}$, where $\mathbb{Q}_h^{\div}$ is the $L^2$ projection to $\bm H_0^h(\div)$, $\bm D_h=\widetilde{\bE}_{h}^{n+1/2}$ and $\bm F_h=\widetilde{\bH}_{h}^{n+1/2}$, we get 
 \begin{align*}
     &\int_{\Omega_{s,h}}\bm B_h^{n+1}\cdot \bm A_h^{n+1}\, \d x -\int_{\Omega_{s,h}}\bm B_h^{n}\cdot \bm A_h^{n}\, \d x\\
     &=2\int_{\Omega_{h}}\l(\bm B_h^{n+1}-\bm B_h^{n} \r)\cdot \widetilde{\bm A}_h^{n+1/2}\, \d x\\
     &=2\int_{\Omega_{h}}\l(\bm B_h^{n+1}-\bm B_h^{n} \r)\cdot \mathbb{Q}_h^{\div}\widetilde{\bm A}_h^{n+1/2}\, \d x\\
        &=-2\Delta t \int_{\Omega_{h}}\nabla\times \bE_h^{n+1/2} \cdot \widetilde{\bm A}_h^{n+1/2}\, \d x\\
        &=-2\Delta t\int_{\Omega_{s,h}}\nabla\times \bE_h^{n+1/2} \cdot \bm A_h^{n+1/2}\, \d x\\
     &=-2\Delta t \int_{\Omega_{s,h}}\bE_{h}^{n+1/2}\cdot  \bB_h^{n+1/2}\, \d x\\
     &=-2\Delta t \int_{\Omega_{h}}\widetilde{\bE}_{h}^{n+1/2}\cdot  \bB_h^{n+1/2}\, \d x\\
      &=-2\Delta t \int_{\Omega_{h}}\widetilde{\bE}_{h}^{n+1/2}\cdot  \bH_h^{n+1/2}\, \d x\\
      &=-2\Delta t\int_{\Omega_{s,h}} \bE_h^{n+1/2} \cdot \bH_h^{n+1/2} \, \d x\\
      &=-2\Delta t\int_{\Omega_{h}} \bE_h^{n+1/2}\cdot \widetilde{\bH}_h^{n+1/2} \, \d x\\
      &=2\Delta t \tau\int_{\Omega_{h}} (\bm j_h^{n+1/2}\times \bm H_h^{n+1/2})\times \bm{H}^{n+1/2}_h\cdot \widetilde{\bm{H}}^{n+1/2}_h \, \d x\\
       &=2\Delta t\tau \int_{\Omega_{s,h}}  (\bm j_h^{n+1/2}\times \bm H_h^{n+1/2})\times \bm{H}^{n+1/2}_h\cdot \bm{H}^{n+1/2}_h \, \d x=0.
 \end{align*}
\end{proof}

\smallskip

\begin{remark}
The discrete Arnold inequality can be proven for any magnetic closed subdomain, i.e.
\begin{equation}
    \l|\int_{\Omega_{s,h}}\bm B_h^{n}\cdot \bm A_h^{n}\, \d x\r| \le C\int_{\Omega_{s,h}}\bm B_h^{n}\cdot \bm B_h^{n}\, \d x.
\end{equation}
Without qualification, by the Arnold inequality we mean the global version, i.e., the above inequality with $\Omega_{s,h} = \Omega$.
\end{remark}

\section{The Lagrange multiplier scheme}\label{sec:LM}
In this part, we introduce a scheme that preserves the global helicity by using a Lagrange multiplier. This approach is inspired by a family of structure-preserving methods for gradient systems \cite{cheng2020new,cheng2020global}, the Klein--Gordon--Schr\"odinger system \cite{guo2023mass}, the geometric evolution equation \cite{garcke2025structure}, a two-phase Stokes model \cite{garcke2025structure_twophase}, and incompressible flows based on finite element exterior calculus \cite{tonnon2024semi}. The key idea is to incorporate scalar variables with energy/helicity variational terms and the evolution equation will reduce to continuous model under mild conditions. Here the evolution equation can either be chosen for the magnetic field $\bB$ \eqref{eqn:magnetic-advection} or the magnetic potential $\bA$ \eqref{eqn:lm-derive}. We choose the latter since it will lead to a discrete scheme that does not violate the discrete Arnold inequality (see \cref{thm:SP} below).

We therefore reformulate \eqref{eqn:magnetic-advection} in terms of the magnetic potential $\bm A$:
\begin{equation}\label{eqn:lm-derive}
    \partial_t\bm A+ \bm E= \bm 0.
\end{equation}
The energy and helicity can be rewritten in terms of $\bm A$ as
\begin{equation}
    \cE=\cE(\bm A)=(\nabla \times \bm  A,\nabla \times \bm  A ),\qquad \cH=\cH(\bm A)=(\nabla \times \bm  A,\bm A ).
\end{equation}

\subsection{Model derivation}
We introduce two Lagrange multipliers into \eqref{eqn:lm-derive} to enforce the energy law and helicity conservation. This yields the following evolution equation
\begin{equation}
    \partial_t\bm A+ \bm E+\lambda_{\cE}\frac{\delta \cE}{\delta \bm A}+\lambda_{\cH}\frac{\delta \cH}{\delta \bm A}  = \bm 0.
\end{equation}
Using Woltjer's variational principle (see \Cref{app:vp}) \cite{woltjer1958theorem}, we have 
\begin{equation}
   \frac{\delta \cE}{\delta \bm A} = 2 \bm j ,\qquad  \frac{\delta \cH}{\delta \bm A} = 2 \bm B,
\end{equation}
where $\bm j=\nabla\times \bm B=\nabla\times\nabla\times \bm A$.

This leads to the continuous PDE system with variables 
 $(\bm B, \bm A,\bm E,\bm j, \lambda_{\cE},\lambda_{H})$
\begin{subequations}\label{eqn:lm-con}
\begin{align}
\partial_t\bm A+ \bm E + 2\lambda_{\cH}\bm B + 2\lambda_{\cE}\bm j&=\bm  0, \label{eqn:lm-con1} \\
 \bm B& = \nabla\times \bm A, \label{eqn:lm-con2}\\
\bm E + \tau(\bm j\times \bm B)\times\bm B &= \bm 0, \label{eqn:lm-con4}\\
\bm j&=\nabla\times \bm B,\label{eqn:lm-con5}\\
\frac{\d}{\d t}\mathcal{E} &= -2\tau \|\bm j\times \bm B\|^2, \label{eqn:lm-con6}\\
\frac{\d}{\d t}\mathcal{H} &=0. \label{eqn:lm-con7}
\end{align}
\end{subequations}
The last two scalar equations enforce the energy law and helicity conservation. The Lagrange multipliers reduce to zero under mild conditions and 
 this model is equivalent to \eqref{eqn:magneto-frictional}.

 \smallskip

 \begin{theorem}\label{thm:equivalent}
    Assume that $(\bm B, \bm  A, \bm E, \bm j)$ and $(\lambda_{\cE}, \lambda_{\cH})$ are the solution of system \eqref{eqn:lm-con}. Then $\lambda_{\cE} = \lambda_{\cH} = 0$ provided $\bm B$ and $\bm j$ are
linearly independent in $L^2(\Omega)$ (equivalently, $\bm j \neq c\,\bm B$
for every constant $c \in \mathbb{R}$).
\end{theorem}

\begin{proof}
    Testing the evolution equation \eqref{eqn:lm-con1}  with $ \bm j$, we obtain 
\begin{align*}
    0= (\partial_t\bm A,\bm j) + (\bm E,\bm j) +2\lambda_{\cE}\|\bm j\|^2+ 2\lambda_{\cH}(\bm B,\bm j).
\end{align*}
With \eqref{eqn:lm-con2}--\eqref{eqn:lm-con6} and straightforward computation, the first two terms give 
\begin{align*}
    (\partial_t\bm A,\bm j) + (\bm E,\bm j)
    &= (\partial_t\bm A,\nabla\times  \bm B)-\tau((\bm j\times \bm B)\times \bm B, \bm j)\\
    &=(\partial_t\bm B, \bm B)+\tau \|\bm j\times \bm B\|^2\\
    &=\frac{1}{2}\l(\frac{\d}{\d t}\mathcal{E} +2\tau \|\bm j\times \bm B\|^2 \r)=0,
\end{align*}
where in the last line we used the enforced energy law. Thus we have
\begin{equation}
    2\lambda_{\cE}\|\bm j\|^2+ 2\lambda_{\cH}(\bm B,\bm j)=0.
\end{equation}
On the other hand, testing \eqref{eqn:lm-con1} with $\bm B$, we obtain 
\begin{align*}
    0= (\partial_t\bm A,\bm B) + (\bm E,\bm B) +2\lambda_{\cE}(\bm j, \bm B)+ 2\lambda_{\cH}\|\bm B\|^2.
\end{align*}
Then using \eqref{eqn:lm-con2}, \eqref{eqn:lm-con4} and \eqref{eqn:lm-con7}, we get 
\begin{align*}
 (\partial_t\bm A,\bm B) + (\bm E,\bm B)&=\frac{1}{2}\frac{\d \cH}{\d t}=0.
\end{align*}
Therefore, we derive another equation for Lagrange multipliers 
\begin{equation}
    2\lambda_{\cE}(\bm j, \bm B)+ 2\lambda_{\cH}\|\bm B\|^2=0.
\end{equation}
To summarize, the two Lagrange multipliers satisfy
    \begin{equation}\label{eq:lm-system}
        \begin{pmatrix}
            \|\bm B\|^2 & (\bm j, \bm B) \\
            (\bm j, \bm B) & \|\bm j\|^2 
        \end{pmatrix}
        \begin{pmatrix}
            \lambda_{\cH}\\
            \lambda_{\cE}
        \end{pmatrix}=
        \begin{pmatrix}
            0\\
            0
        \end{pmatrix}.
    \end{equation}
    Thus, by the Cauchy--Schwarz inequality,
\[
   \|\bm B\|^2\|\bm j\|^2 - (\bm j,\bm B)^2 \ge 0,
\]
with equality iff $\bm j = c\,\bm B$ a.e.\ for a single constant $c$
(or $\bm B = \bm 0$). Hence the determinant is strictly positive
whenever $\bm B$ and $\bm j$ are linearly independent in $L^2(\Omega)$,
and in that case the Lagrange multipliers vanish.
\end{proof}

\begin{corollary}
Assume the solution considered in \Cref{thm:equivalent} is continuous in time. Then the magneto-friction equation \eqref{eqn:magneto-frictional} is equivalent to \eqref{eqn:lm-con} in the sense that any solution to \eqref{eqn:magneto-frictional} combined with $\lambda_{\cE} = \lambda_{\cH} = 0$ solves \eqref{eqn:lm-con}, and any solution to \eqref{eqn:lm-con} satisfies $\lambda_{\cE} = \lambda_{\cH} = 0$ and thus solves \eqref{eqn:magneto-frictional}.
\end{corollary}

\begin{proof}
     \Cref{thm:equivalent} implies that before reaching a stationary state, the Lagrange multipliers in  \eqref{eqn:lm-con} vanish. Thus the solutions to the two systems are equivalent. Since the solutions are assumed to be continuous in time, this equivalence also extends to stationary states.
\end{proof}

\subsection{Full discretization}

We discretize \eqref{eqn:lm-con} using finite element exterior calculus in space and implicit Euler time stepping.

\begin{problem}[Lagrange multiplier scheme]
   For each time step $n\ge 0$, we find $$(\bm B^{n+1}_h, \bm A^{n+1}_h,\bm  E^{n+1}_h, \bm j^{n+1}_h)\in \bm H_0^h(\div) \times [\bm H_0^h(\curl)]^3$$  and $ (\lambda_{\cE,h}^{n+1},\lambda_{\cH,h}^{n+1})$\  $\in \mathbb{R}^2$ such that for $(\bm C_h, \bm D_h,\bm F_h,\bm k_h)\in \bm H_0^h(\div)\times [\bm H_0^h(\curl)]^3$,
\begin{subequations}\label{eqn:lm-full}
\begin{align}\label{eqn:lm-full1}
\begin{split}
    &\l(\frac{\bm A_h^{n+1} - \bm A_h^n}{\Delta t}, \bm D_h \r)+ (\bm E_h^{n+1}, \bm D_h) \\
    & \qquad+ \lambda_{\cH,h}^{n+1}(\bm B_h^{n+1}, \bm D_h)+\lambda_{\cE,h}^{n+1}(\bm j^{n+1}_h,  \bm D_h) = 0,
\end{split}
\end{align}
\vspace{-13pt}
 \begin{align}
 \qquad\qquad\qquad\ (\bm B^{n+1}_h, \bm C_h)-(\nabla\times \bm A^{n+1}_h, \bm C_h) &=0  \label{eqn:lm-full2}, \\
(\bm E^{n+1}_h,\bm  F_h) + \tau((\bm j_h^{n+1}\times \bm B_h^{n+1})\times \bm B_h^{n+1},\bm  F_h) &= 0,\label{eqn:lm-full4}\\
(\bm j_h^{n+1}, \bm k_h)- (\bm B_h^{n+1}, \nabla\times \bm k_h)&=0,\label{eqn:lm-full5}\\
\frac{\mathcal{E}_h^{n+1}- \mathcal{E}_h^{n}}{\Delta t} + 2\tau \|\bm B^{n+1}_h\times \bm j^{n+1}_h\|^{2} &= 0,\label{eqn:lm-full6}\\
\mathcal{H}^{n+1}_h-\mathcal{H}^{n}_h &= 0.\label{eqn:lm-full7}
\end{align}
\end{subequations}
\end{problem}

\smallskip

\begin{theorem}\label{thm:SP}
    Let $(\bm B_h^{n+1}, \bm A_h^{n+1},\bm  E_h^{n+1}, \bm j_h^{n+1})$ and $ (\lambda_{\cE,h}^{n+1},\lambda_{\cH,h}^{n+1})$ be a solution of \eqref{eqn:lm-full}. The energy law \eqref{energy-law} and the discrete Gauss law hold. Moreover, the global helicity is conserved:
     \begin{equation}
            \cH^{n+1}_h= \cH^n_h= \ldots = \cH^0_h,\qquad n\ge 0,
        \end{equation}
        and the discrete Arnold's inequality holds.
\end{theorem}

\begin{proof}
      The energy-decreasing and the helicity-preserving properties are direct consequences of \eqref{eqn:lm-full6} and \eqref{eqn:lm-full7} respectively. Moreover, from \eqref{eqn:lm-full2}, we have
    \[\bm B^{n+1}_h=\nabla\times \bm A^{n+1}_h.\]
    Taking the divergence, we get $\nabla\cdot \bm B^{n+1}_h=0$. Finally, the discrete Arnold inequality follows from  \eqref{eqn:lm-full2} and the discrete Poincar\'e inequality
    \begin{align*}
        |\cH^{n+1}_h|=|(\bm A^{n+1}_h,\bm B^{n+1}_h)|
        &\le \|\bm A^{n+1}_h\|\|\bm B^{n+1}_h\|\\
        &\le C\|\nabla\times \bm A^{n+1}_h\|\|\bm B^{n+1}_h\|\\
        &\le C\|\bm B^{n+1}_h\|^2=C\mathcal{E}_h^{n+1},
    \end{align*}
	    where $C$ is a positive constant independent of $n$.
\end{proof}

\smallskip
\begin{remark}
  As the system approaches a force-free equilibrium, $\bm j\times\bm B\to\bm 0$
  and the enforced dissipation $2\tau\|\bm j\times\bm B\|^2\to 0$, so the energy
  constraint \eqref{eqn:lm-full6} loses sensitivity to $\lambda_{\cE}$ and the
  multiplier becomes ill-determined, making the Lagrange multiplier scheme
  \eqref{eqn:lm-full} difficult to converge near the steady state. In practice, the Lagrange
  multiplier approach must be modified to support long-time evolution \cite{garcke2025structure,garcke2025structure_twophase,
  cheng2020global,cheng2020new}. If the discrete energy dissipation rate
  $\Delta \mathcal{E}_h^{n+1}
  =\frac{\mathcal E_h^n-\mathcal E_h^{n+1}}{\Delta t}$ is greater than $\gamma$, where $\gamma\ll 1$, we continue with the two-multiplier method as presented. Otherwise, we set $\lambda_{\cE} = 0$ and omit \eqref{eqn:lm-full6}, and only preserve global
  helicity. In the subsequent numerical experiments, we always set $\gamma=9\times 10^{-5}$. In addition, to improve robustness before this threshold is reached,
  we perform the same strategy if the dissipation rate increases for two consecutive time steps, subject to the relative tolerance $10^{-4}$.
  \end{remark}

\subsection{Solver}
For each time step, we apply Newton linearization for the coupled system \eqref{eqn:lm-full}. The linearized Newton system for the unknowns 
\begin{equation}
    \left(\bm B^{n+1}_h,\bm A^{n+1}_h, \bm E^{n+1}_h, \bm j^{n+1}_h,
\lambda_{\mathcal E,h}^{n+1}, \lambda_{\mathcal H,h}^{n+1}\right)
\end{equation}
naturally has a saddle point structure due to the presence of the Lagrange multipliers. We group the degrees of freedom into a vector of \emph{physical fields},
\begin{equation}
x_{\mathrm f} 
= (\bm B, \bm A, \bm E, \bm j),    
\end{equation}
and a vector of \emph{Lagrange multipliers},
\begin{equation}
x_{\lambda} = (\lambda_{\mathcal E}, \lambda_{\mathcal H}),    
\end{equation}
which yields a $2\times2$ block system
\begin{equation}
\label{eq:block-system}
\begin{pmatrix}
A & B \\
C & 0
\end{pmatrix}
\begin{pmatrix}
x_{\mathrm f} \\[2pt]
x_{\lambda}
\end{pmatrix}
=
\begin{pmatrix}
f_{\mathrm f} \\[2pt]
f_{\lambda}
\end{pmatrix}.
\end{equation}
The right-hand side functions $f_{\mathrm{f}}$ and $f_{\mathrm{\lambda}}$ are the residual in the Newton iterations. We solve~\eqref{eq:block-system} with a block preconditioner. We first use flexible GMRES \cite{saad1993flexible} as the outermost solver. We then split the physical fields from the Lagrange multipliers. More precisely, we use the full block factorization preconditioner \cite{murphy2000,ipsen2001}
\begin{equation}
\label{eq:full-schur}
\mathcal{P}^{-1}
=
\begin{pmatrix}
I & 0 \\[4pt]
C A^{-1} & I
\end{pmatrix}
\begin{pmatrix}
A & 0 \\[4pt]
0 & S
\end{pmatrix}
\begin{pmatrix}
I & A^{-1} B \\[4pt]
0 & I
\end{pmatrix},
\end{equation}
where $S = -\,C A^{-1} B$ is the Schur complement. The physical block is then solved by a direct solver; we employ MUMPS~\cite{amestoy2001}. The Schur complement involves only two scalar multipliers. A GMRES with a maximum $2$ iterations is applied, which is sufficient in exact arithmetic. All operations are performed in a matrix-free way, except for the direct solve of $A$. The above solver is implemented using PETSc's built-in Schur complement infrastructure \cite{petsc-user-ref}. 

\section{Numerical experiments}\label{sec:Numerical-experiment}
We simulate magnetic relaxation with two kinds of magnetic configurations, the magnetic braids and the magnetic knots. In particular, we consider two types of magnetic braids, Wilmot-Smith (WS) (\Cref{subsub:WS}) and Candelaresi-Pontin-Hornig (CPH) (\Cref{subsub:CPH}). For magnetic knots, we choose the Hopf fibration (\Cref{subsec:knots}). We do so using the three schemes considered: the non-conservative scheme \eqref{eq:nonconservative-scheme} that preserves no helicity constraint, the projection-based scheme \eqref{eq:mixed-scheme} with local helicity constraints, and the Lagrange multiplier method \eqref{eqn:lm-full} with the global helicity constraint.

The computational domain is chosen to be a cuboid
\begin{equation}
    \Omega = (-4, 4)^2 \times (-Z, Z),
\end{equation}
where $Z = 10$ for the magnetic knots and $Z = 24$ for the magnetic braids. The coupling parameter $\tau$ sets the relaxation rate of the magneto-frictional dynamics ($\tfrac{\d}{\d t}\mathcal{E}=-2\tau\|\bm j\times \bm B\|^2$). It controls the speed of convergence to the equilibrium $\bm j\times \bm B=\bm 0$ but not the equilibrium itself. We set $\tau = 1$ and $T=10000$. 


Unless stated otherwise, we employ a coarse mesh consisting of $8\times 8\times 10$ hexahedral cells for the magnetic knots and $4\times 4\times 24$ hexahedral cells for the magnetic braids in the $x \times y \times z$ directions. For spatial discretization, we use, in each case, the lowest-order N\'ed\'elec edge and face elements of the first kind belonging to the same de~Rham complex.
These under-resolved discretizations stress-test the structure-preserving properties of the algorithms.

The projection of the initial data to the discretely divergence-free subspace of $\bm H_0^h(\div)$ and the evaluation of the helicity are the same as our previous work \cite[Section 4.1]{he2025helicity}.

\subsection{Magnetic braids}\label{subsec:braids}
Different braided configurations can be constructed by prescribing the
locations and signs of the localized twists. In the Wilmot--Smith (WS)
configuration~\cite{wilmot-smith2009}, the local twists have alternating
signs, so that the signed twisting cancels. In the
Candelaresi--Pontin--Hornig (CPH) configuration~\cite{Candelaresi2015},
the local twists have the same sign, so that the twists add coherently.

Although these configurations have been widely used as model braided magnetic
fields, their helicity is not usually computed explicitly in the existing
literature. The main difficulty is that these braid fields are open magnetic
configurations: the magnetic field has non-zero flux through the top and bottom
faces $( \bm B\cdot \bm n \neq 0)$. 
Therefore, the classical magnetic helicity is not directly gauge invariant and is not a suitable helicity invariant without further modification. One standard way to obtain a well-defined helicity for
such open fields is to use relative helicity~\cite{berger1984topological}.
In this work, we consider taking the periodic boundary conditions on the top and bottom faces, then we can use the generalized helicity introduced in our previous
work~\cite{he2025helicity}. For the Hodge decomposition of the full magnetic field 
\begin{equation}
    \bm B_h = \nabla\times \bm A_h + \bm B_H,
\end{equation}
where
\begin{equation}
    \bm B_H = B_0 \mathbf e_z ,
\end{equation}
the discrete generalized helicity is defined by
\begin{equation}
    \tilde{\mathcal H}_h
    :=
	    (\bm A_h, \bm B_h+\bm B_H).
\end{equation}
This quantity is well-defined for the braid configurations considered here and
provides a computable helicity in the finite element setting. 
In practice, for periodic braids, the Lagrange multiplier scheme preserves
\(\widetilde{\mathcal H}_h\) by imposing the scalar constraint on $\tilde{\mathcal H}_h$ whose variation is $\frac{\delta \tilde{\mathcal H}_h}{\delta \bm A_h}=2\bm B_h$. The projection-based scheme instead preserves this quantity through the same
discrete local-helicity mechanism, with the fixed harmonic component included
when evaluating \(\tilde{\mathcal H}_h\).

As shown in the
numerical experiments below, it distinguishes the two choices of local twists:
the alternating-sign WS braid has vanishing generalized helicity, whereas the
same-sign CPH braid has non-zero generalized helicity. In this sense, the generalized helicity detects the coherent twisting in the CPH
braid, while it vanishes for the alternating-sign WS braid. Moreover, this
generalized helicity is associated with a generalized Arnold inequality in the
discrete setting \cite[Theorem~3.8]{he2025helicity}, so that a non-zero value provides a lower bound on the
relaxation energy and hence an effective topological barrier. The harmonic
component remains constant according to \cite[Theorem~3.4]{he2025helicity},
contributing only a constant background energy.

\begin{remark}
   In the present work, we have fixed the boundary conditions of magnetic braids to be periodic for brevity and do not address their influence on the equilibrium structure. Though theoretical existence of equilibria is open for MF, numerical evidence \cite{pontin2016braided,yeates2019magnetohydrodynamic} show that the choice of boundary conditions (line-tied, periodic and closed) can fundamentally alter the nature of braided equilibria. We leave a systematic investigation of this dependence to future work.
\end{remark}
\subsubsection{Wilmot-Smith (WS) configuration}\label{subsub:WS}
The WS configuration for $E^3$-field \cite{wilmot-smith2009,pontin2016braided} is constructed by concatenating three identical elementary units, each consisting of one positive and one negative twist superimposed on a uniform background field. Therefore, the global helicity vanishes by construction. 
Note that Arnold's inequality does not guarantee a topological barrier for the $E^3$-field and thus a physically faithful simulation is even more challenging (than initial fields with non-zero helicity). The initial magnetic configuration is
\begin{align}
\bm B^{\text{WS}}(0)
= B_0 \mathbf{e}_z
&+ \sum_{c=1}^{6} \frac{2 kk_c B_0}{a} 
\left( -(y - y_c)\,\mathbf{e}_x + (x - x_c)\,\mathbf{e}_y \right)
\times\\ &\qquad\qquad \exp\left[
 -\frac{(x - x_c)^2}{a^2}
 -\frac{(y - y_c)^2}{a^2}
 -\frac{(z - z_c)^2}{l^2}
\right],\notag
\end{align}
with the initial field strength $B_0$, strength of twist $k$, radius and length in the $z$-direction of the twist region $a$ and $l$, respectively. The twist locations are $(x_c, y_c, z_c)$. Take $x_c = k_c = \{1, -1, 1, -1, 1, -1\}
,
y_c = \{0, \dots, 0\},
z_c = \{-20, -12, -4, 4, 12, 20\},
$ and
$a = \sqrt{2},
l = 2,
B_0 = 1$, $k=5$. For the time stepping, we choose $\Delta t = 0.1$ at the beginning of $200$ time steps and then change to a larger time step $\Delta t = 100$. This empirical adjustment improves the robustness of the nonlinear solver while allowing us to reach the long final time more efficiently. We monitor the total energy, the generalized helicity $\tilde{\mathcal{H}}_h$ and the background energy $\|\bm B_H\|^2$. 
\begin{figure}[htpb]
\centering
\renewcommand{\arraystretch}{1.9}
\def\temptablewidth{2.9\textwidth}
\setlength{\tabcolsep}{6pt}
\begin{tabular}{p{0.16\textwidth}| c c}
\hline
Schemes & $\mathcal{E}_h$, $\widetilde{\mathcal{H}}_h$ & Errors \\ \hline
\vspace{-3cm} Non-conservative scheme \eqref{eq:nonconservative-scheme} with no constraint
& \includegraphics[width=0.33\textwidth]{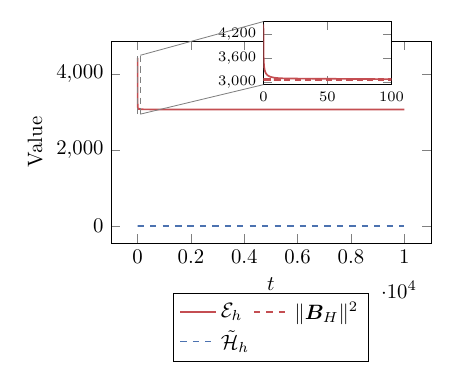}
& \includegraphics[width=0.33\textwidth]{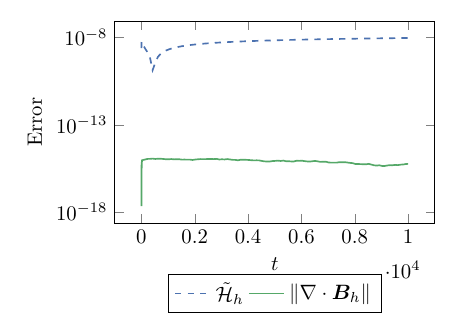} \\
\hline
\vspace{-3cm} Lagrange multiplier scheme \eqref{eqn:lm-full} with global constraint
& \includegraphics[width=0.33\textwidth]{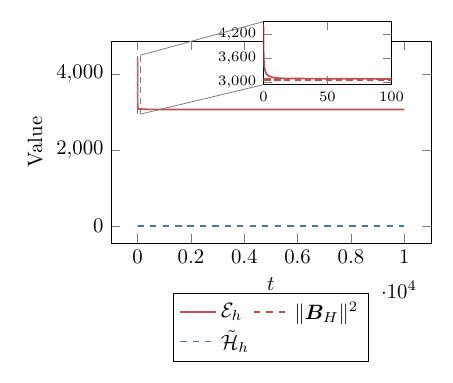}
& \includegraphics[width=0.33\textwidth]{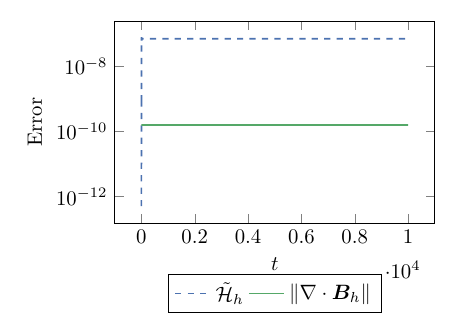} \\
\hline
\vspace{-3cm} Projection-based scheme \eqref{eq:mixed-scheme} with local constraints
& \includegraphics[width=0.33\textwidth]{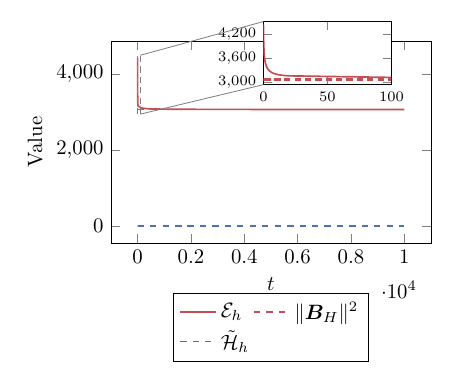}
& \includegraphics[width=0.33\textwidth]{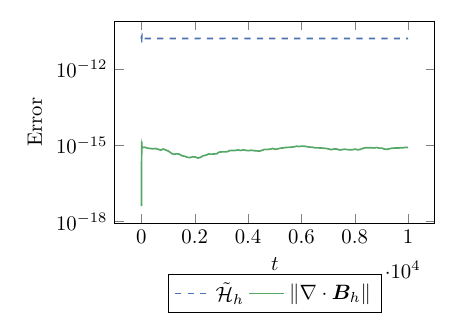}\\
\hline
\end{tabular}
\caption{Magnetic braids (WS): evolution of the total energy, generalized helicity, and errors $|Q-Q(0)|$ for $\tilde{\mathcal{H}}_h$ and $\|\nabla\cdot  \bm B_h\|$.}
\label{fig:e3-alternating}
\end{figure}

\begin{figure}[htbp]
\centering
\renewcommand{\arraystretch}{1.7}
\setlength{\tabcolsep}{2pt}
\begin{tabular}{p{1.8cm}<{\centering}| c c c }
\hline
   Schemes & $t=0$ & $t=10$ & $t=10000$ \\ \hline
 \vspace{-5cm} Non-con\-ser\-va\-tive scheme \eqref{eq:nonconservative-scheme} with no constraint
    & \includegraphics[scale=0.135]{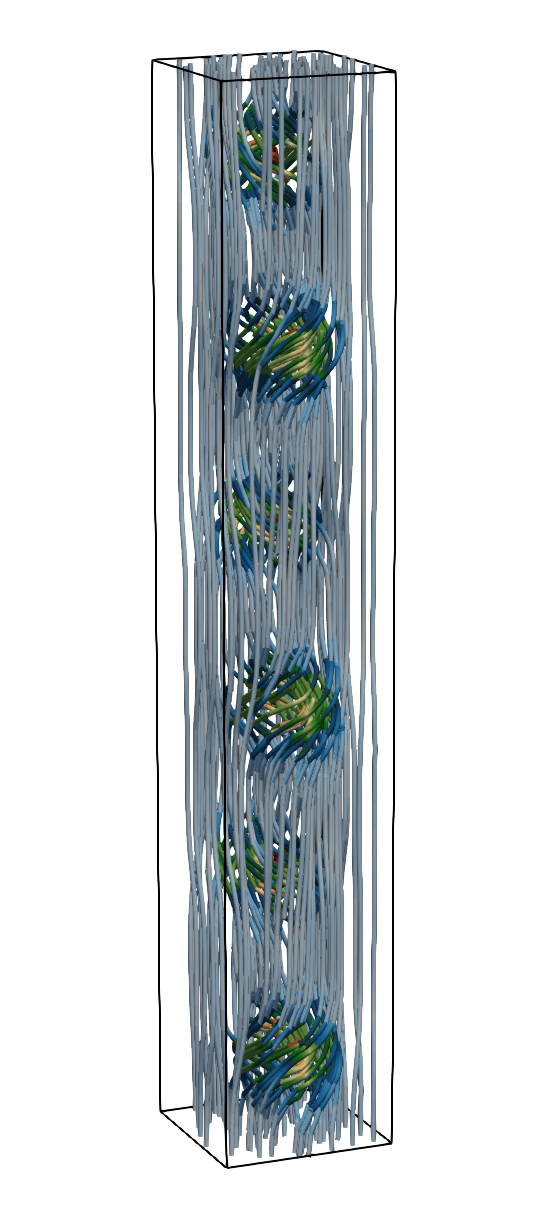}
    &\hspace{-0.3cm} \includegraphics[scale=0.135]{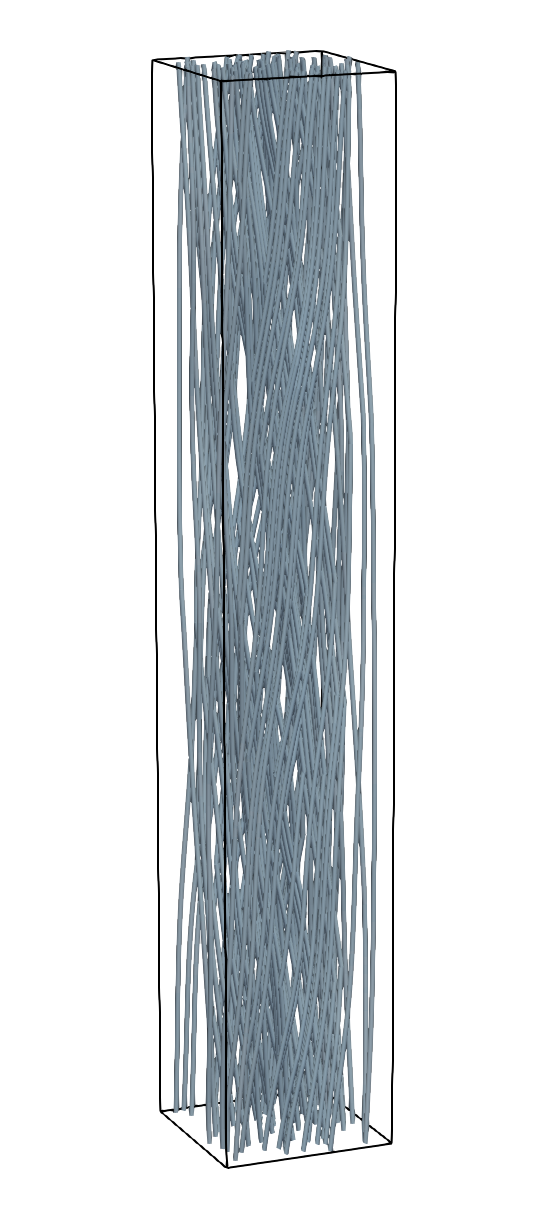}
    &\hspace{-0.3cm} \includegraphicswithlegend[0.83]{.24\textwidth}{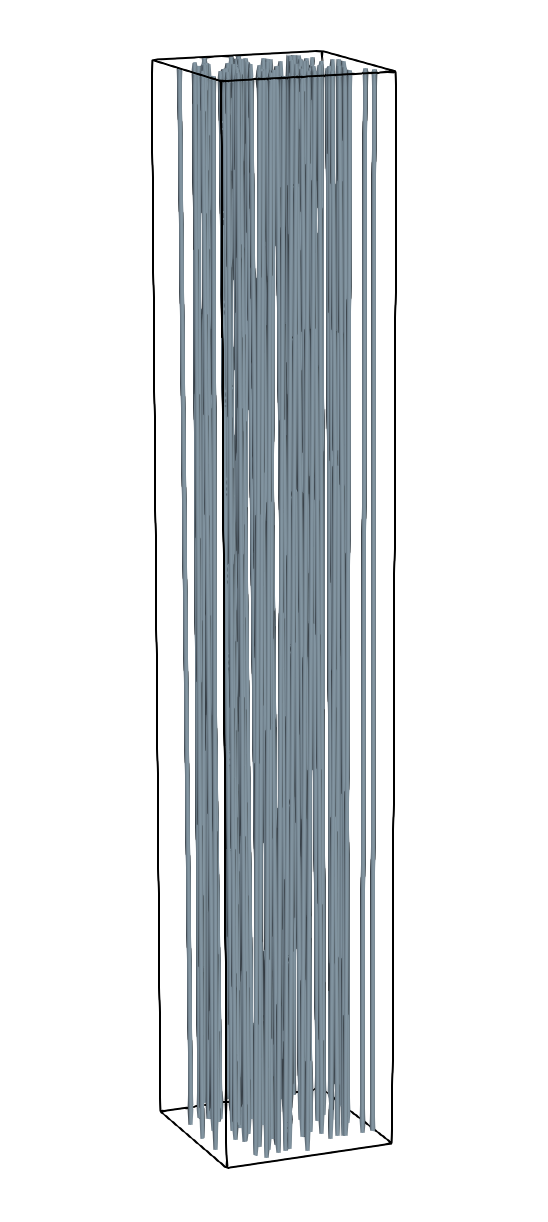}{3.1}{0.97}  \\ \hline
    \vspace{-5cm} Lagrange multiplier scheme \eqref{eqn:lm-full} with global constraint
    & \includegraphics[scale=0.135]{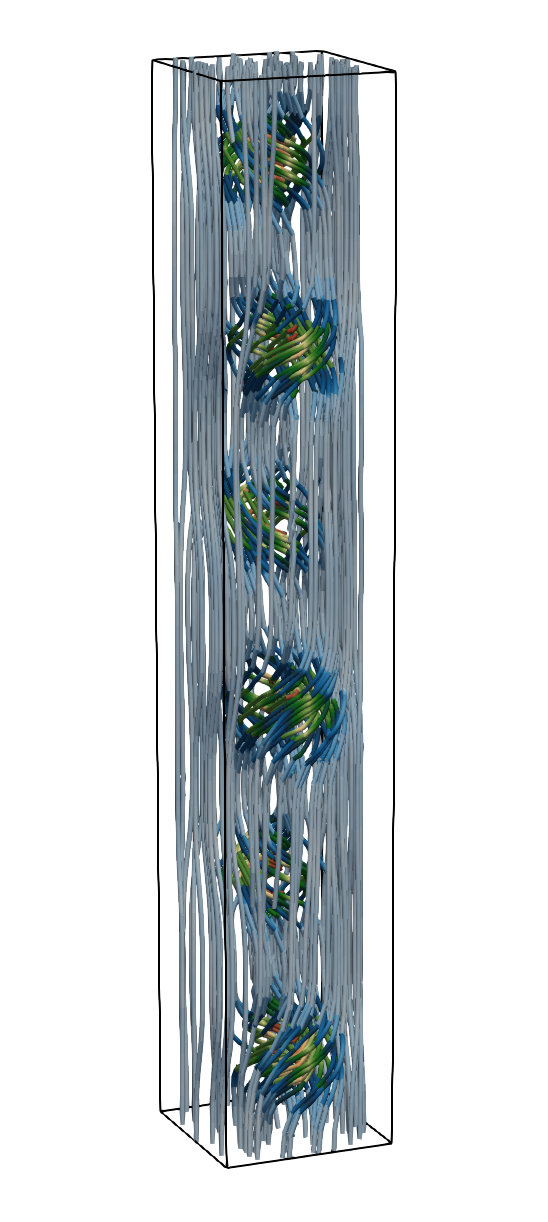}
    &\hspace{-0.3cm} \includegraphics[scale=0.135]{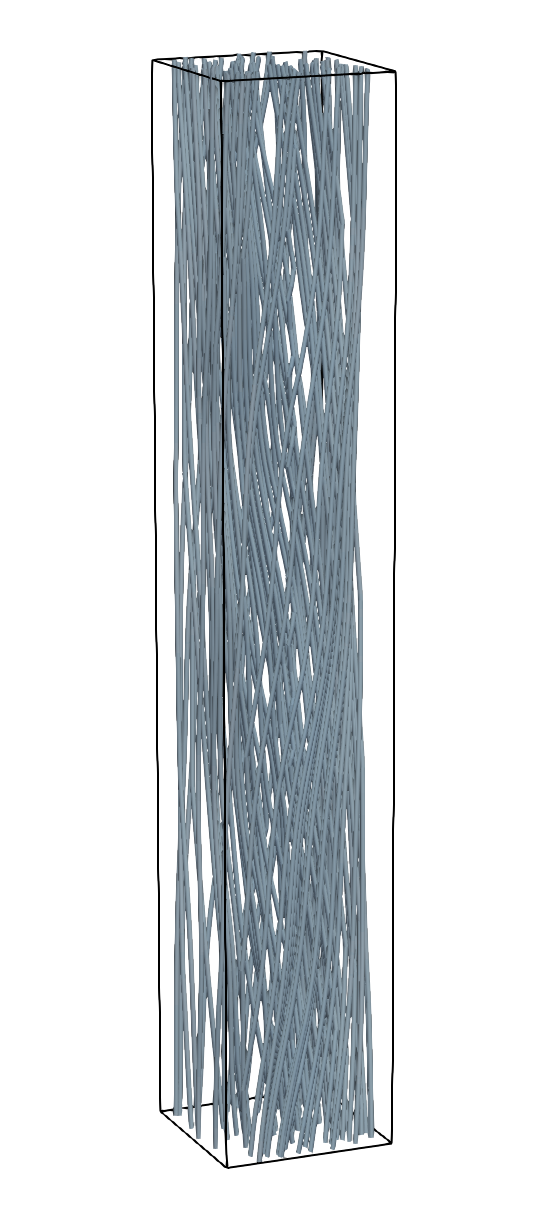}
    &\hspace{-0.3cm} \includegraphicswithlegend[0.86]{.23\textwidth}{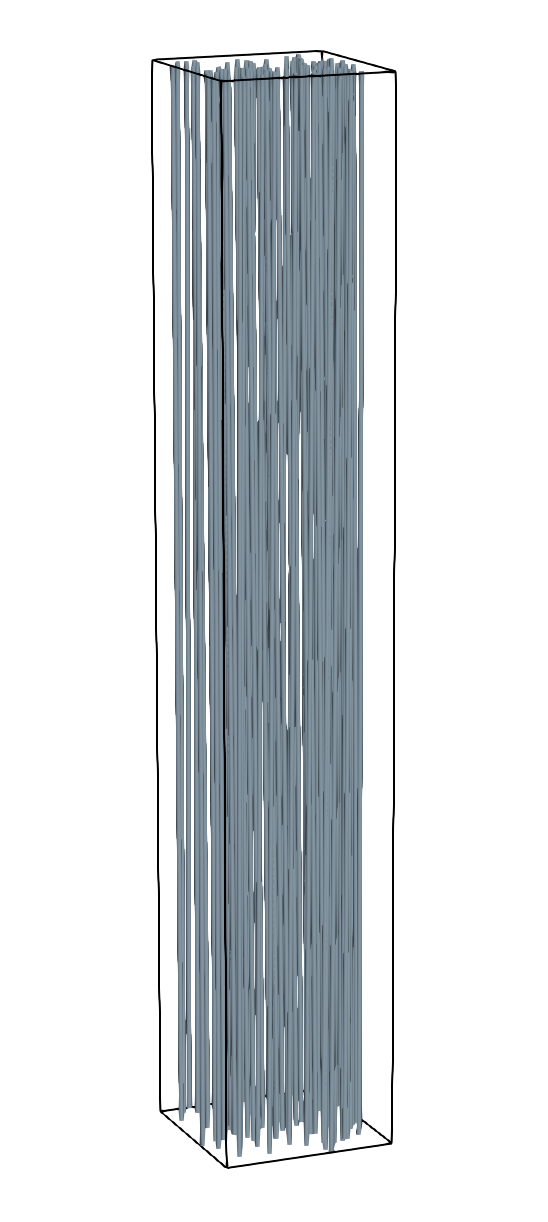}{3.1}{0.97} \\ \hline
     \vspace{-5cm} Projection-based scheme \eqref{eq:mixed-scheme} with local constraints
    &  \includegraphics[scale=0.135]{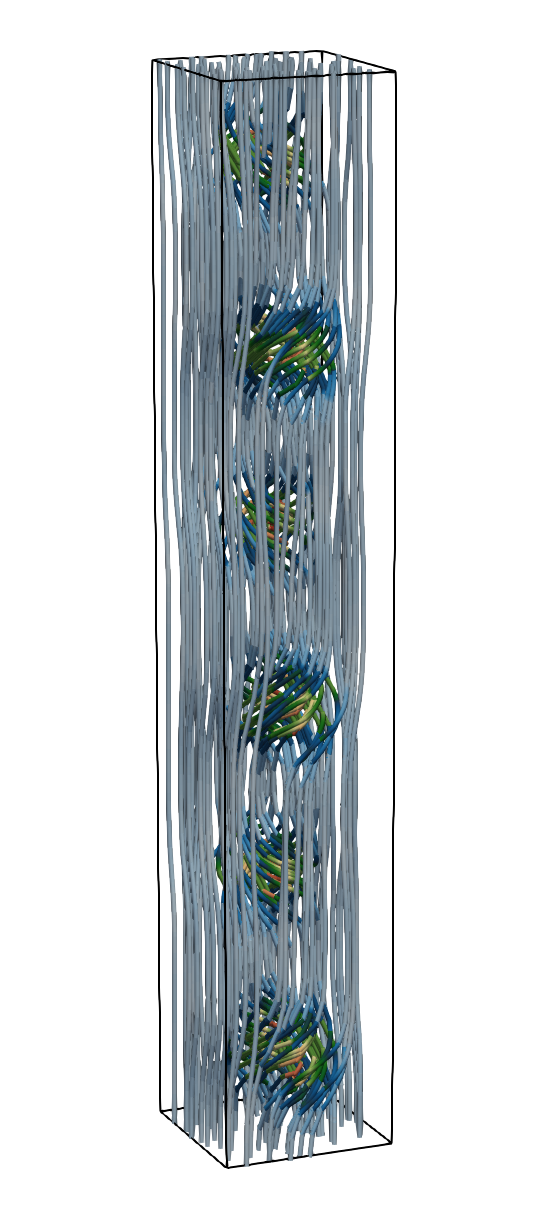}
    &\hspace{-0.3cm} \includegraphics[scale=0.135]{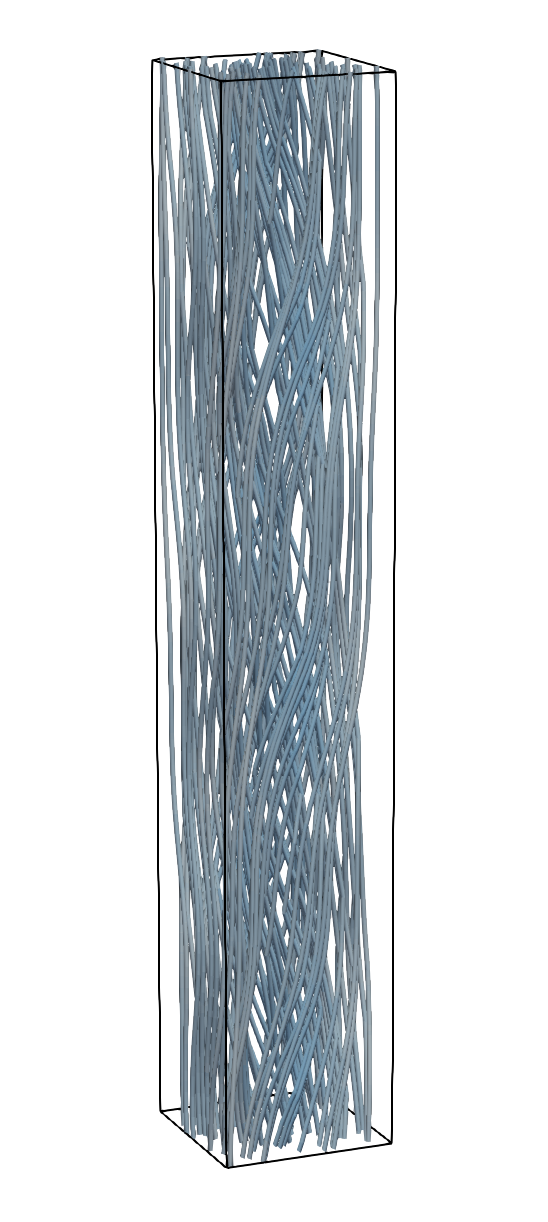}
    &\hspace{-0.3cm} \includegraphicswithlegend[0.86]{.23\textwidth}{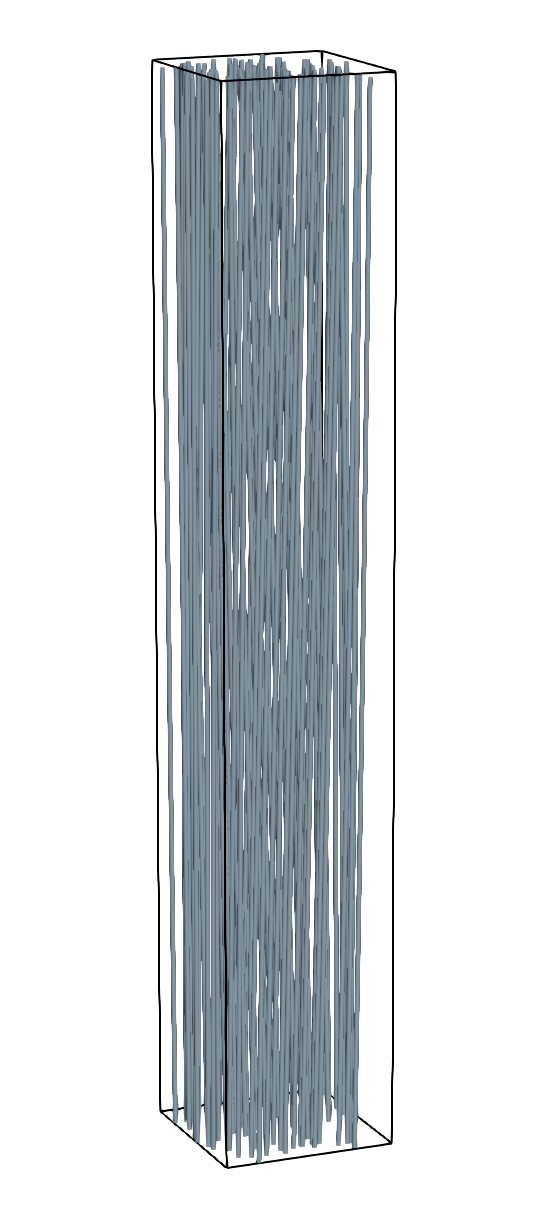}{3.1}{0.97} \\ \hline
\end{tabular}
\caption{\emph{Magnetic braids} (WS): comparison of evolution of stream tubes of the magnetic field under different topological constraints, colored by magnetic field strength $\|\bm B_h\|$.}
\label{fig:e3-alter-simulation}
\end{figure}

\begin{figure}[htpb]
    \centering
    \includegraphics[width=0.8\linewidth]{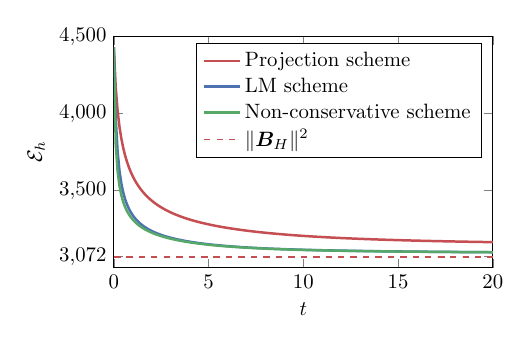}
    \caption{Zoom-in relaxation path of (WS): the projection scheme tends to be slower than the other two schemes.}
    \label{fig:plot-e3-comparison}
\end{figure}

\Cref{fig:e3-alternating} illustrates the evolution of the total energy, generalized helicity, and errors. All three schemes relax to a 
trivial uniform background state but in different paths. \Cref{fig:e3-alter-simulation} demonstrates that magnetic reconnection occurs in all three schemes, so that the
braided structures gradually untangle and the magnetic field relaxes toward a nearly
rectilinear, unbraided configuration. 

Since Arnold’s inequality gives no positive lower bound for the magnetic energy, in this case, the Lagrange multiplier method behaves much closer to an
unconstrained relaxation, whereas the projection-based method follows a slower relaxation pathway. (see \Cref{fig:plot-e3-comparison}). This example shows that preserving helicity, even at the local discrete level,
does not by itself guarantee preservation of nontrivial braided topology during
the relaxation.

\subsubsection{Candelaresi-Pontin-Hornig (CPH) configuration}\label{subsub:CPH}
We change the local twist to all positive as in \cite{Candelaresi2015}. The initial magnetic configuration is 
\begin{align}
\bm B^{\text{CPH}}(0)
= B_0 \mathbf{e}_z
&+ \sum_{c=1}^{6} \frac{2 k B_0}{a} 
\left( -(y - y_c)\,\mathbf{e}_x + (x - x_c)\,\mathbf{e}_y \right)
\times\\ &\qquad\qquad \exp\left[
 -\frac{(x - x_c)^2}{a^2}
 -\frac{(y - y_c)^2}{a^2}
 -\frac{(z - z_c)^2}{l^2}
\right],\notag
\end{align}
with the initial field strength $B_0$, strength of twist $k$, radius and length in the $z$-direction of the twist region $a$ and $l$, respectively. The twist locations are $(x_c, y_c, z_c)$. We choose $x_c = \{1, -1, 1, -1, 1, -1\}
,
y_c = \{0, \dots, 0\},
z_c = \{-20, -12, -4, 4, 12, 20\},
$ and
$a = \sqrt{2},
l = 2,
B_0 = 1$, $k=5$. For the time stepping, we choose $\Delta t = 0.1$ at the beginning of $200$ time steps and then change to a larger time step $\Delta t = 100$. We monitor the total energy and the generalized helicity. 

\begin{figure}[htpb]
\centering
\renewcommand{\arraystretch}{1.9}
\def\temptablewidth{2.9\textwidth}
\setlength{\tabcolsep}{6pt}
\begin{tabular}{p{0.16\textwidth}| c c}
\hline
Schemes & $\mathcal{E}_h$, $\tilde{\mathcal{H}}_h$ & Errors \\ \hline
\vspace{-3cm} Non-conservative scheme \eqref{eq:nonconservative-scheme} with no constraint
& \includegraphics[width=0.33\textwidth]{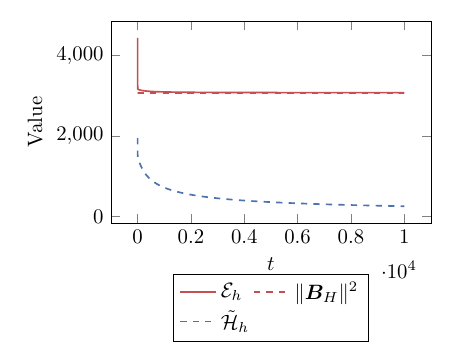}
& \includegraphics[width=0.33\textwidth]{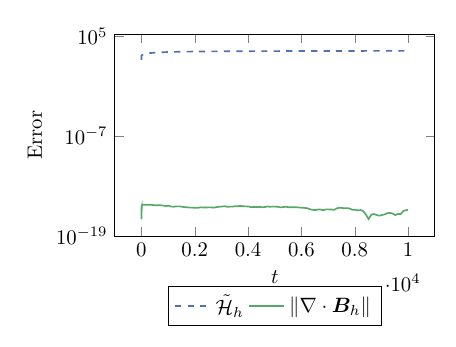} \\ \hline
  \vspace{-3cm} Lagrange multiplier scheme \eqref{eqn:lm-full} with global constraint
& \includegraphics[width=0.33\textwidth]{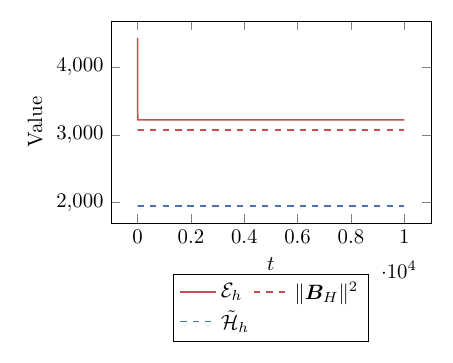}
& \includegraphics[width=0.33\textwidth]{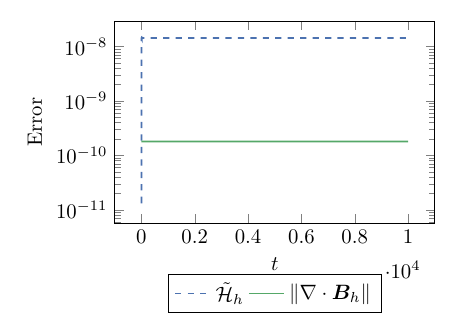} \\ \hline
 \vspace{-3cm} Projection-based scheme \eqref{eq:mixed-scheme} with local constraints
& \includegraphics[width=0.33\textwidth]{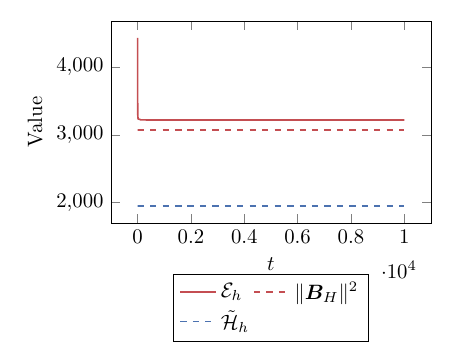}
& \includegraphics[width=0.33\textwidth]{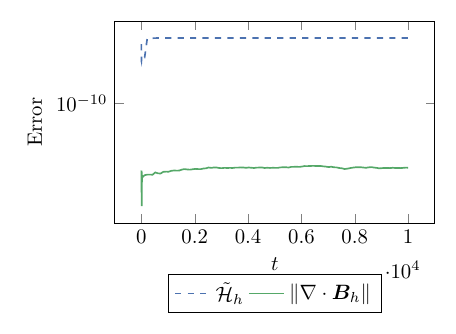}\\ \hline
\end{tabular}
\caption{\emph{Magnetic braids} (CPH): evolution of total energy, background energy $\|\bm B_H\|^2$, generalized helicity, and errors $|Q-Q(0)|$ for $\tilde{\mathcal{H}}_h$ and $\|\nabla\cdot  \bm B_h\|$.}
\label{fig:e3-positive-eh-errors}
\end{figure}

\begin{figure}[htbp]
\centering
\renewcommand{\arraystretch}{1.7}
\setlength{\tabcolsep}{2pt}
\begin{tabular}{p{1.8cm}<{\centering}| c c c }
\hline
   Schemes & $t=0$ & $t=10$ & $t=10000$ \\ \hline
 \vspace{-5cm} Non-con\-ser\-va\-tive scheme \eqref{eq:nonconservative-scheme} with no constraint
    & \includegraphics[scale=0.135]{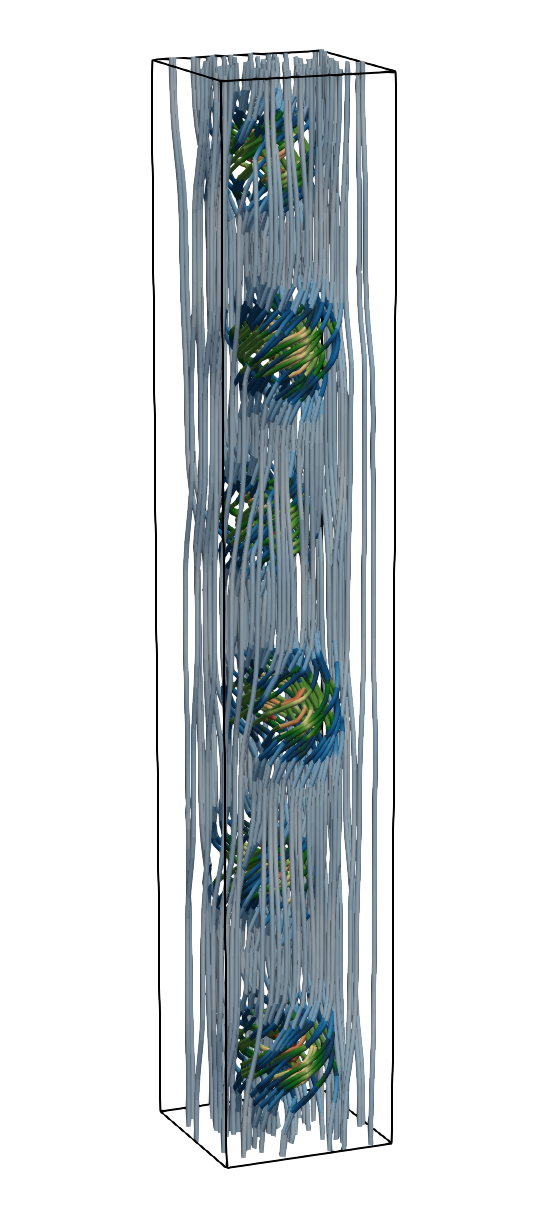}
    &\hspace{-0.3cm} \includegraphics[scale=0.135]{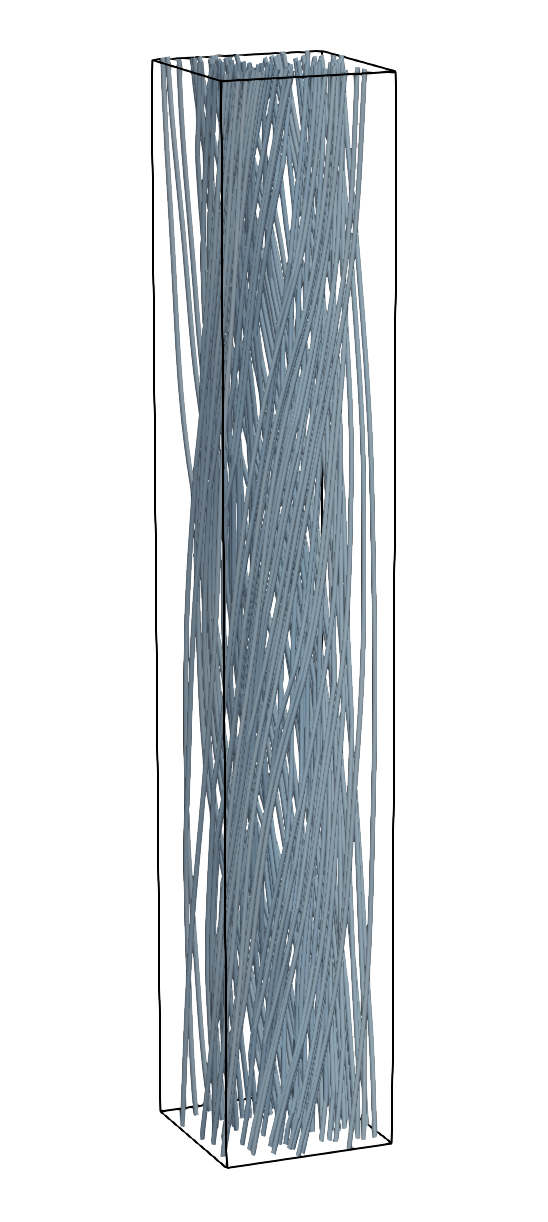}
    &\hspace{-0.3cm} \includegraphicswithlegend[0.83]{.24\textwidth}{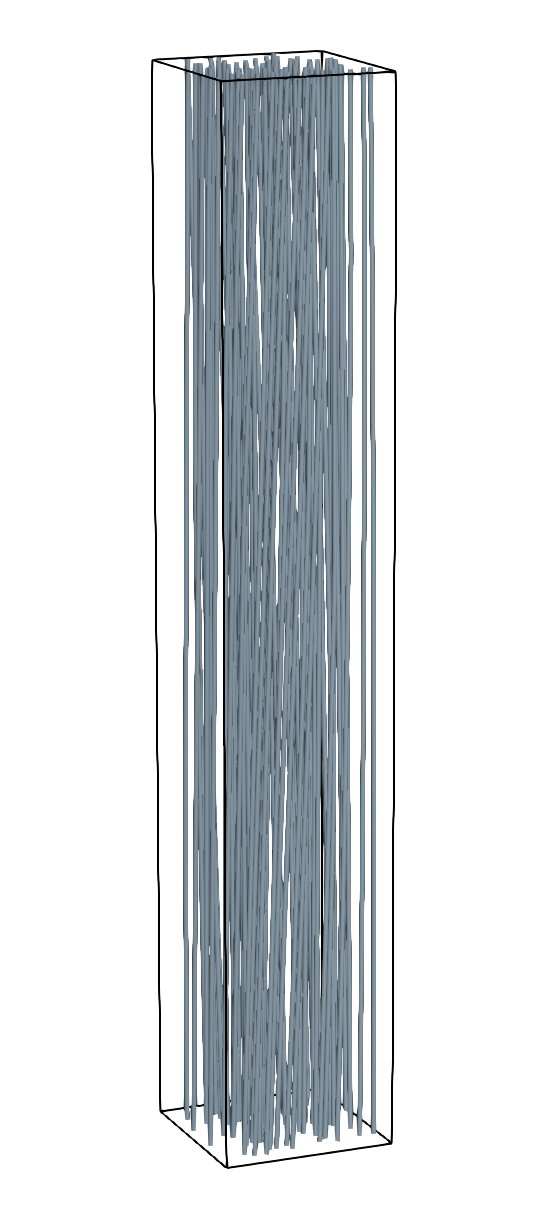}{3.1}{0.97}  \\ \hline
    \vspace{-5cm} Lagrange multiplier scheme \eqref{eqn:lm-full} with global constraint
    & \includegraphics[scale=0.135]{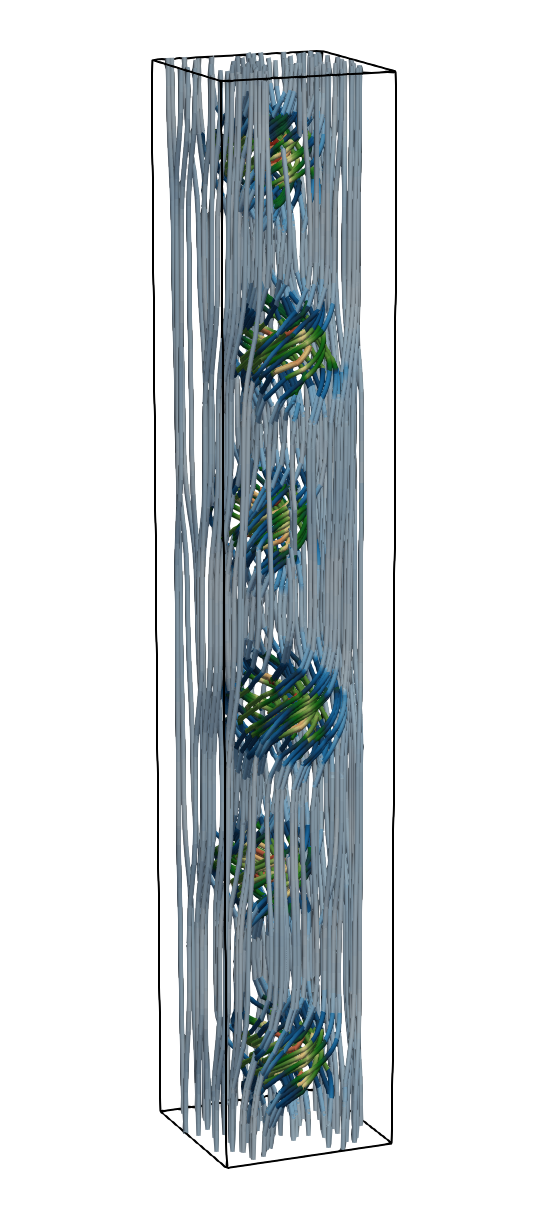}
    &\hspace{-0.3cm} \includegraphics[scale=0.135]{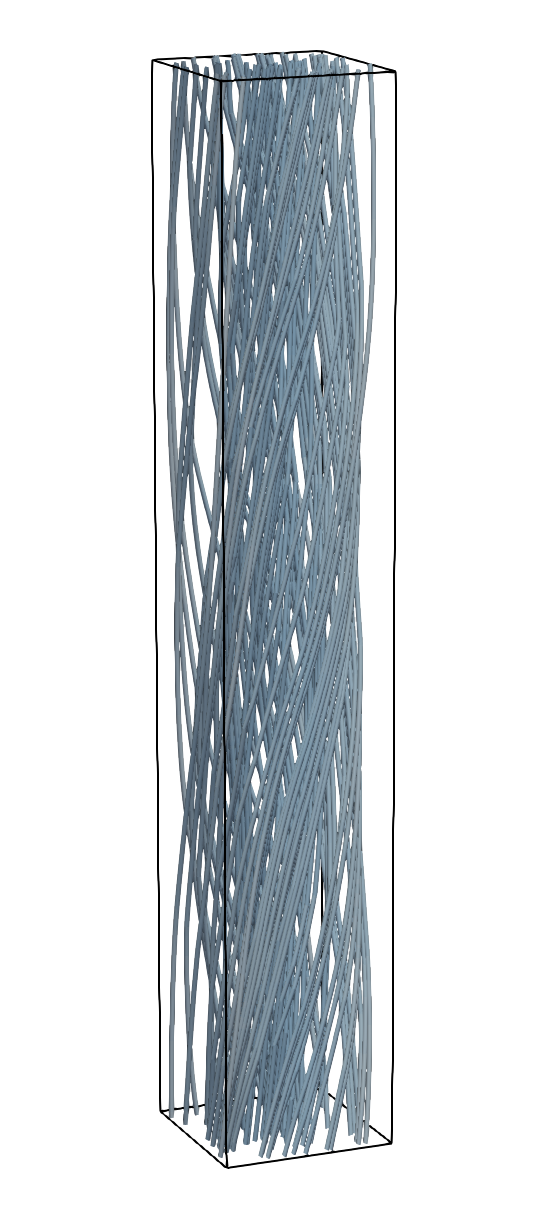}
    &\hspace{-0.3cm} \includegraphicswithlegend[0.86]{.23\textwidth}{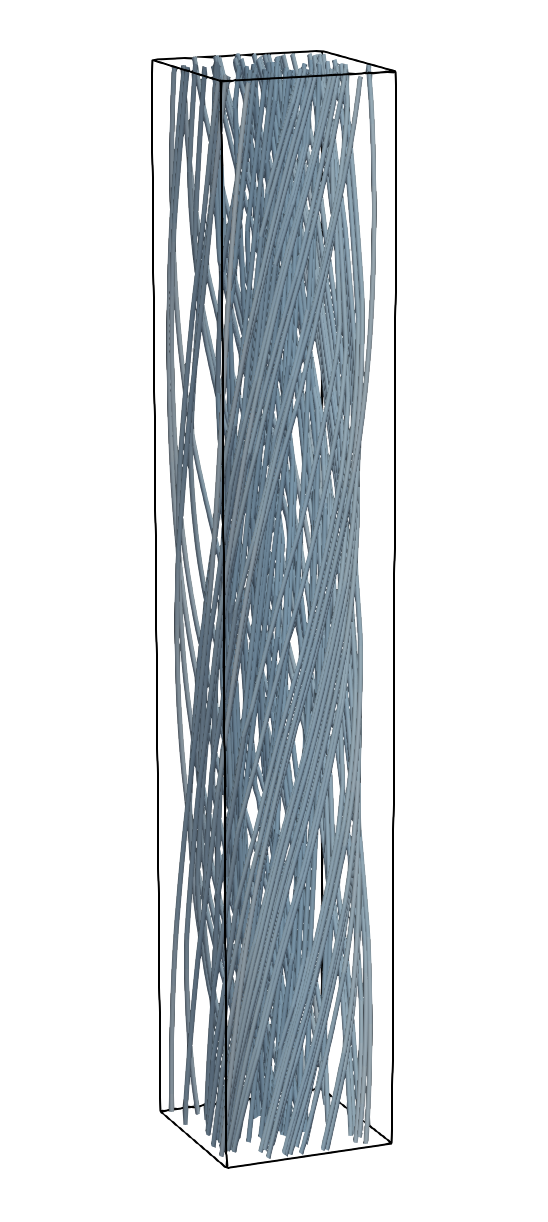}{3.1}{0.97} \\ \hline
     \vspace{-5cm} Projection-based scheme \eqref{eq:mixed-scheme} with local constraints
    &  \includegraphics[scale=0.135]{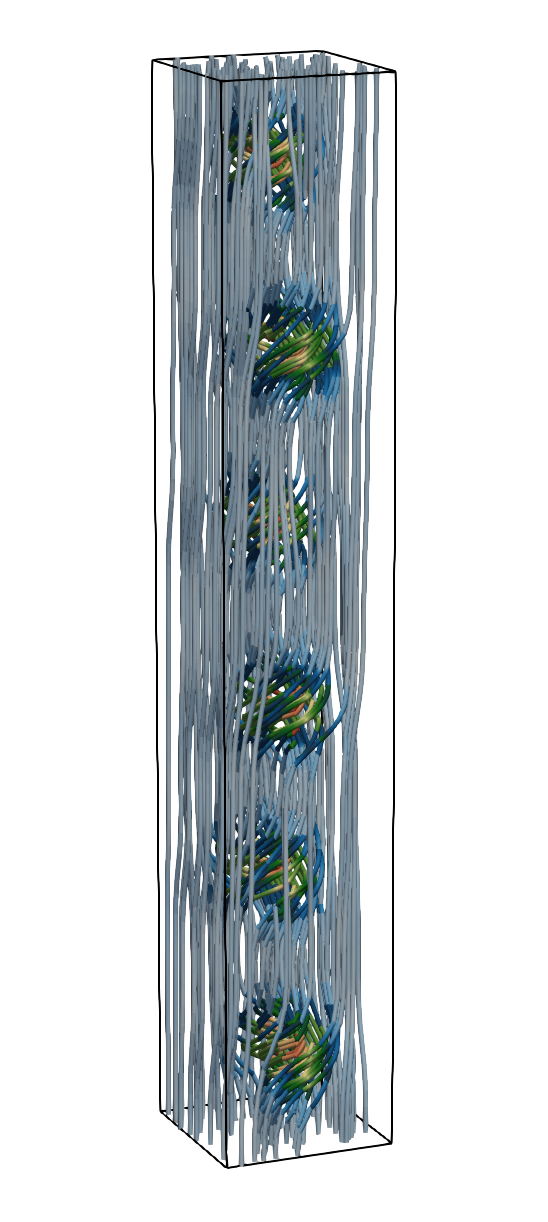}
    &\hspace{-0.3cm} \includegraphics[scale=0.135]{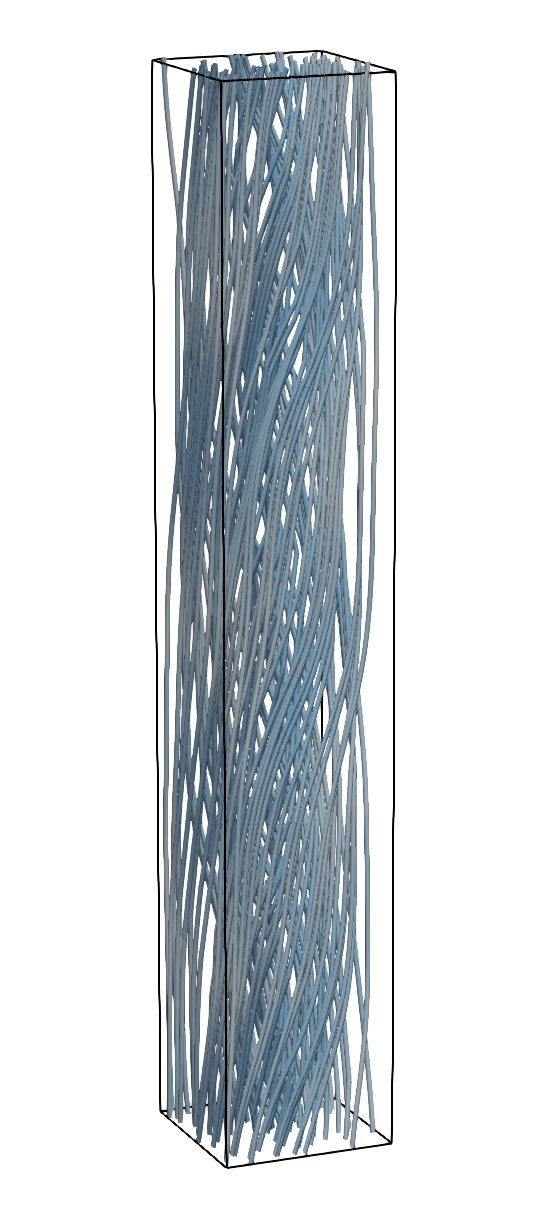}
    &\hspace{-0.3cm} \includegraphicswithlegend[0.86]{.23\textwidth}{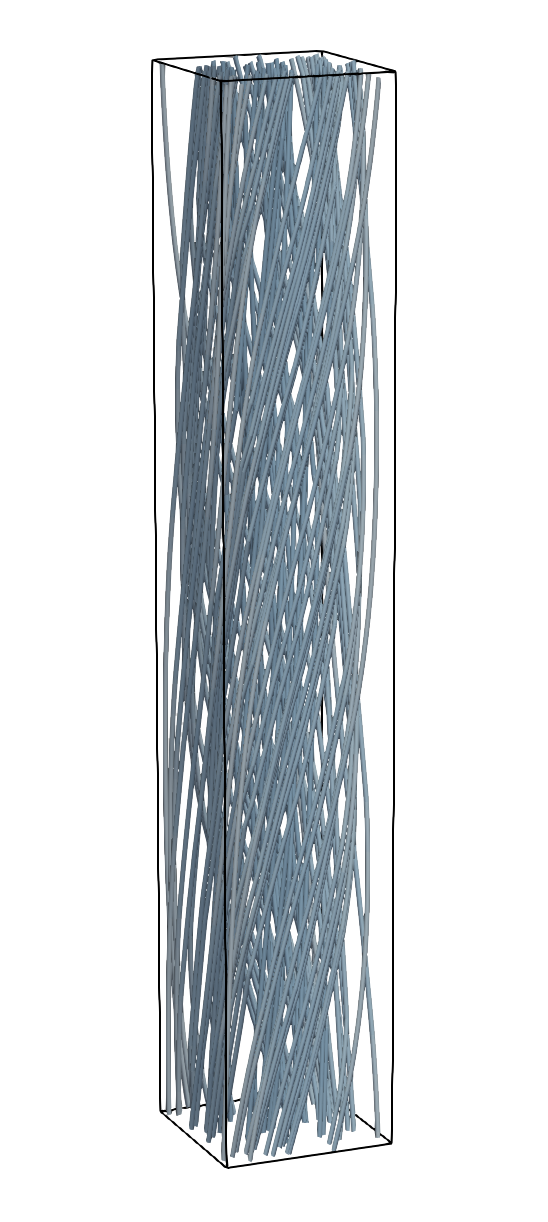}{3.1}{0.97} \\ \hline
\end{tabular}
\caption{\emph{Magnetic braids} (CPH): comparison of evolution of stream tubes of the magnetic field under different topological constraints, colored by magnetic field strength $\|\bm B_h\|$.}
\label{fig:e3-positive-simulation}
\end{figure}

\Cref{fig:e3-positive-eh-errors,fig:e3-positive-simulation} illustrate the evolution of the 
total energy, background energy, generalized helicity, and errors for 
the different schemes. We observe that the generalized helicity is non-zero for the CPH configuration, 
in contrast to the WS configuration considered above. This confirms that the 
generalized helicity provides an accurate and computable helicity measure for magnetic braids. 

During the relaxation process, the non-conservative scheme relaxes to the background field due to the helicity pollution. By contrast, the Lagrange multiplier scheme and the projection-based scheme both preserve
the generalized helicity, and therefore do not relax to the background field in
the same way as the non-conservative scheme. This confirms that the
generalized helicity can act as an effective topological invariant for this class
of braided fields.

\subsection{Magnetic knots: Hopf fibration}\label{subsec:knots}
We next consider the relaxation of magnetic knots, employing the Hopf fibration as the initial configuration  \cite{smietIdealRelaxationHopf2017}
\begin{equation}
    \bm{B}^{\mathrm{Hopf}}(0)
    =
    \frac{4 \sqrt{s}}{\pi\left(1+r^2\right)^3 \sqrt{\omega_1^2+\omega_2^2}}\left(\begin{array}{c}
        2\left(\omega_2 y-\omega_1 x z\right) \\
        -2\left(\omega_2 x+\omega_1 y z\right) \\
        \omega_1\left(-1+x^2+y^2-z^2\right)
    \end{array}\right),
    \label{eqn:hopf-fibre}
\end{equation}
where $\omega_1, \omega_2 \in \mathbb R$ are winding numbers, $s \ge 0$ is a scaling parameter, and $r^2 = x^2 + y^2 + z^2$. We choose $\omega_1 = 3$, $\omega_2 = 2$, $s = 1$, such that the field lines form three windings in the poloidal direction for every two in the toroidal direction, thus exhibiting a non-zero helicity. For the time stepping, we choose $\Delta t = 1$ at the beginning of $200$ time steps and then change to a larger time step $\Delta t = 100$. Again, this empirical adjustment improves the robustness of the nonlinear solver while allowing us to reach the long final time more efficiently. We choose the boundary conditions to be Dirichlet on all faces. 

\Cref{fig:hopf-simulation} presents snapshots of the relaxation process with the three schemes. The equilibria reached at $t=10000$ clearly exhibit distinct morphological features depending on the constraints enforced. This qualitative difference is further quantified by the time evolution of key physical quantities shown in \Cref{fig:hopf-eh-errors}.

\begin{figure}[h!]
\centering
\renewcommand{\arraystretch}{1.9}
\def\temptablewidth{2.9\textwidth}
\setlength{\tabcolsep}{6pt}
\begin{tabular}{p{0.16\textwidth}| c c}
\hline
Schemes & $\mathcal{E}_h$, $\mathcal{H}_h$ & Errors \\ \hline
\vspace{-3cm} Non-conservative scheme \eqref{eq:nonconservative-scheme} with no constraint
& \includegraphics[width=0.33\textwidth]{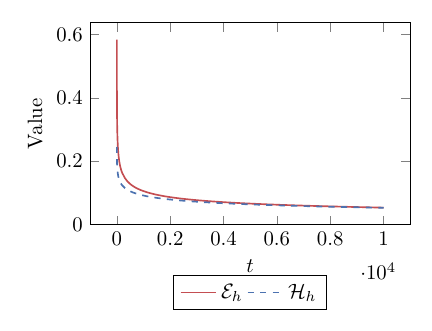}
& \includegraphics[width=0.33\textwidth]{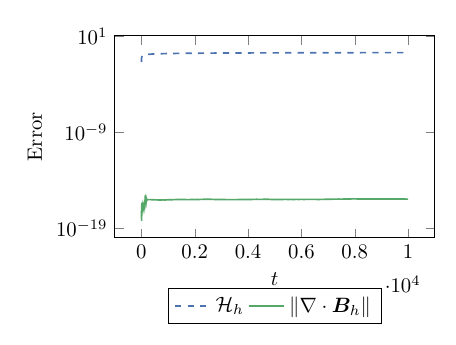} \\
\hline
\vspace{-3cm} Lagrange multiplier scheme \eqref{eqn:lm-full} with global constraint
& \includegraphics[width=0.33\textwidth]{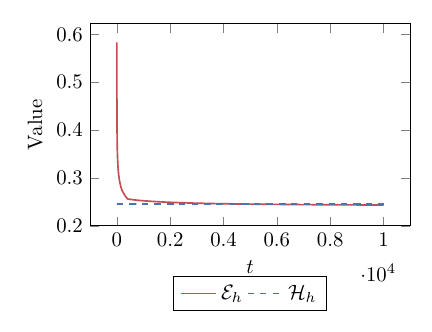}
& \includegraphics[width=0.33\textwidth]{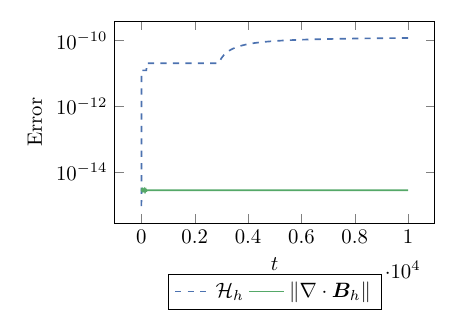} \\
\hline
\vspace{-3cm} Projection-based scheme \eqref{eq:mixed-scheme} with local constraints
& \includegraphics[width=0.33\textwidth]{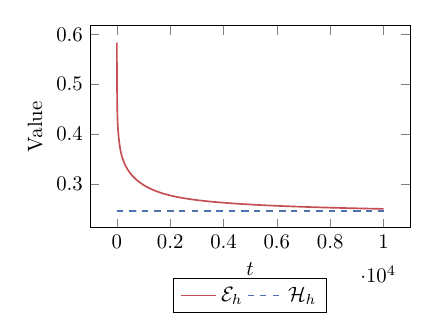}
& \includegraphics[width=0.33\textwidth]{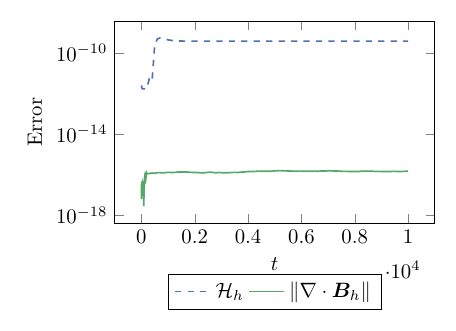}\\
\hline
\end{tabular}
\caption{Magnetic knots (Hopf fibration): evolution of energy and helicity, errors.}
\label{fig:hopf-eh-errors}
\end{figure}

\begin{figure}[h!]
\centering
\renewcommand{\arraystretch}{1.9}
\def\temptablewidth{2.9\textwidth}
\setlength{\tabcolsep}{6pt}
\begin{tabular}{p{1.89cm}<{\centering}  |c c c }
\hline
    Schemes &  $t=0$ & $t=10$ & $t=10000$ \\ \hline
  \vspace{-4.5cm} Non-conservative scheme \eqref{eq:nonconservative-scheme} with no constraint
    & \includegraphics[width=0.23\textwidth]{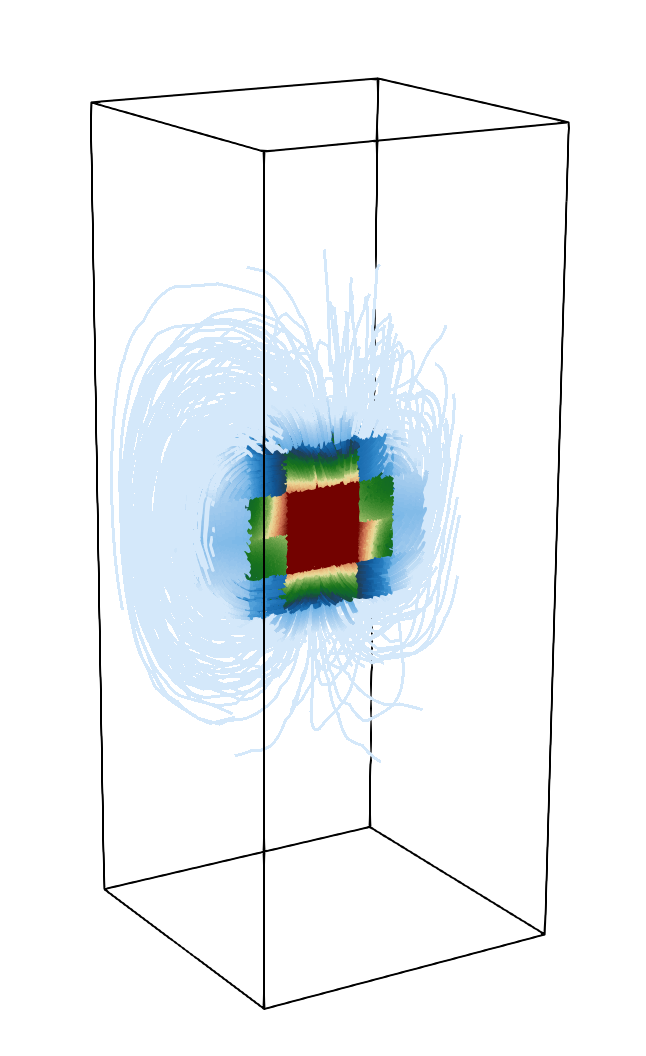}
    & \includegraphics[width=0.23\textwidth]{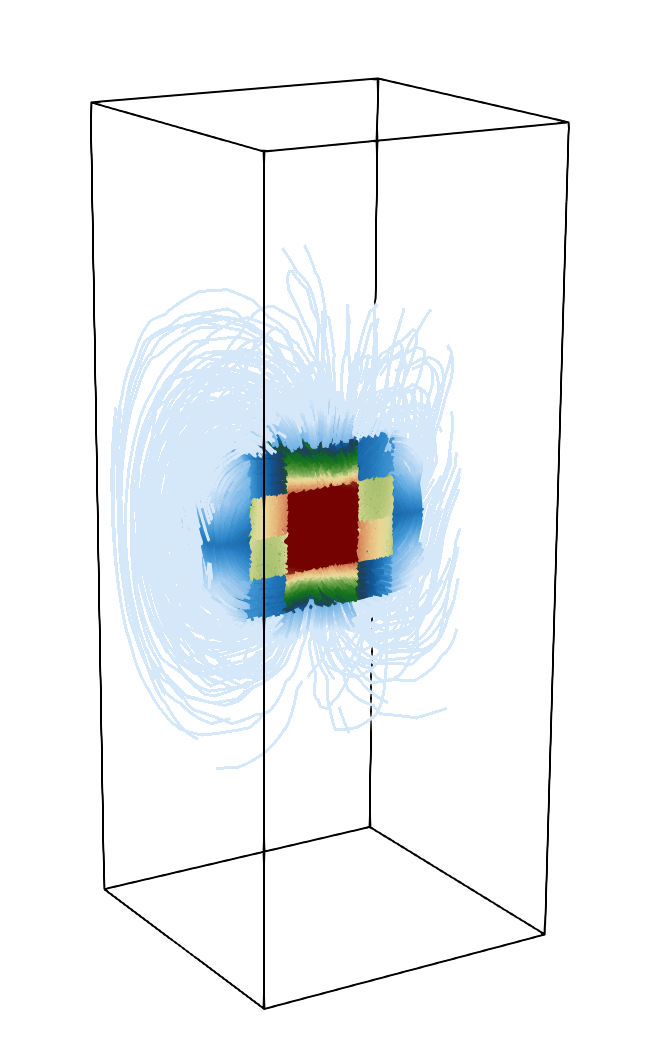}
    &  \includegraphicswithlegend[0.85]{.26\textwidth}{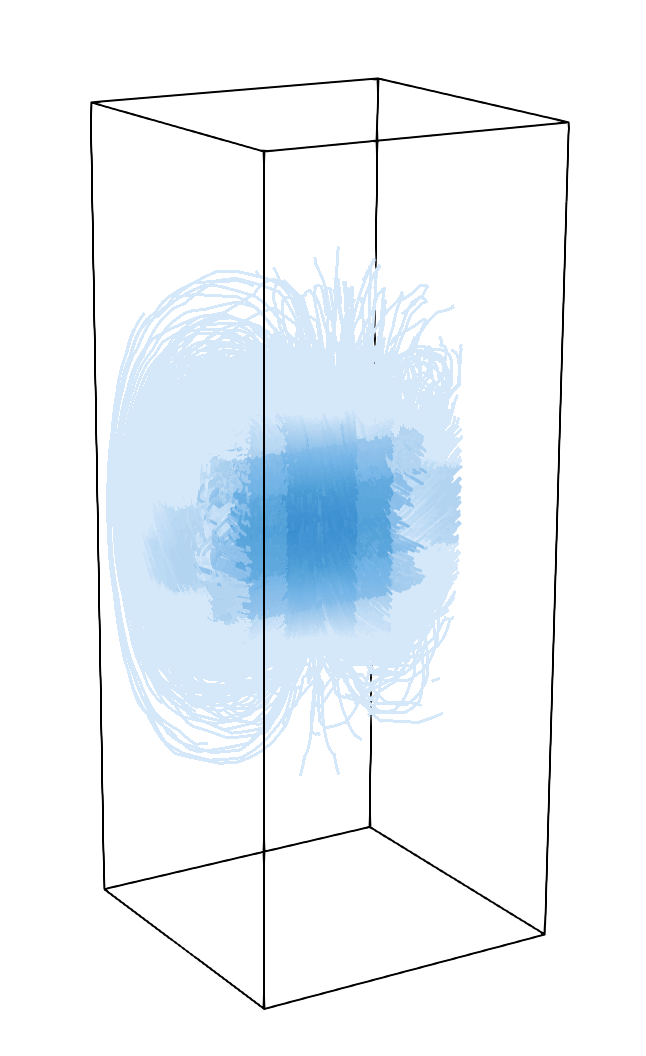}{0.33}{0}  \\
    \hline
     \vspace{-4.5cm} Lagrange multiplier scheme \eqref{eqn:lm-full} with global constraint
    & \includegraphics[width=0.23\textwidth]{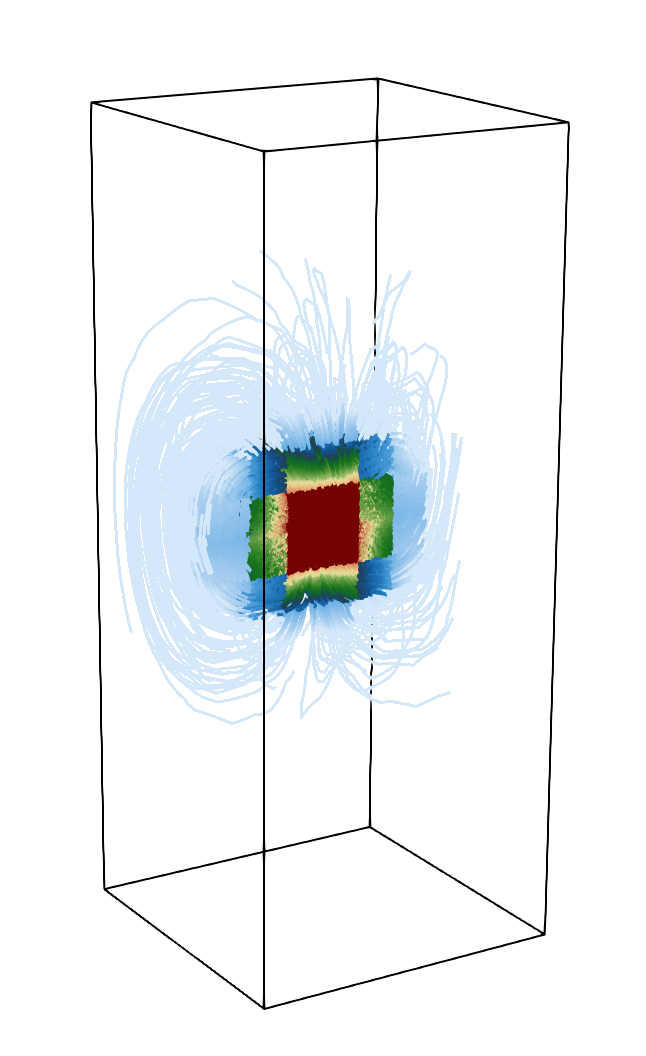}
    & \includegraphics[width=0.23\textwidth]{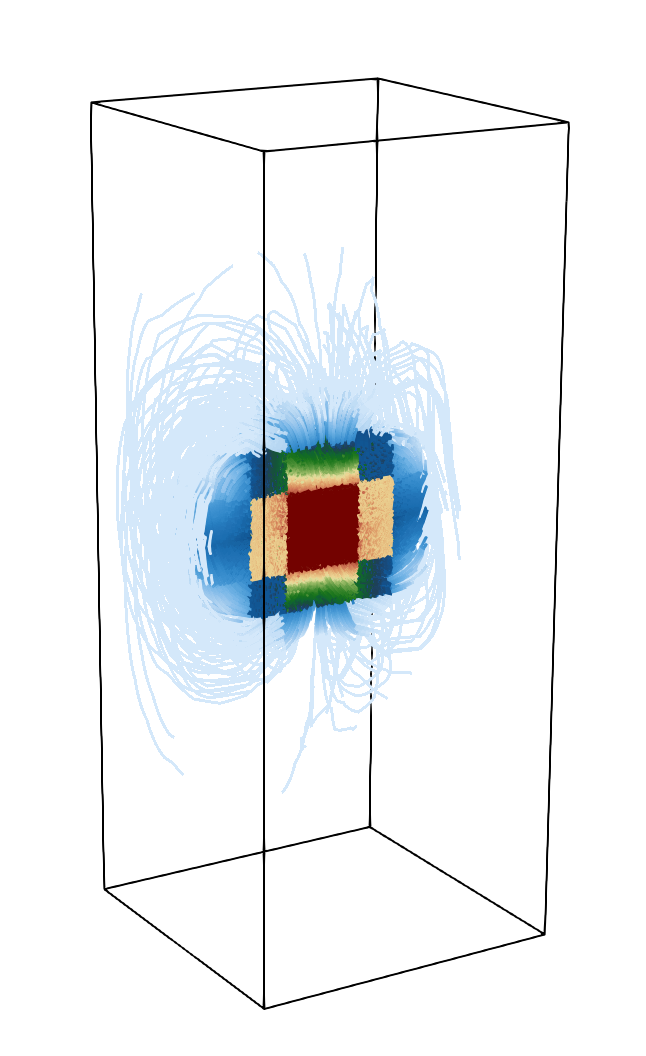}
    & \includegraphicswithlegend[0.85]{.26\textwidth}{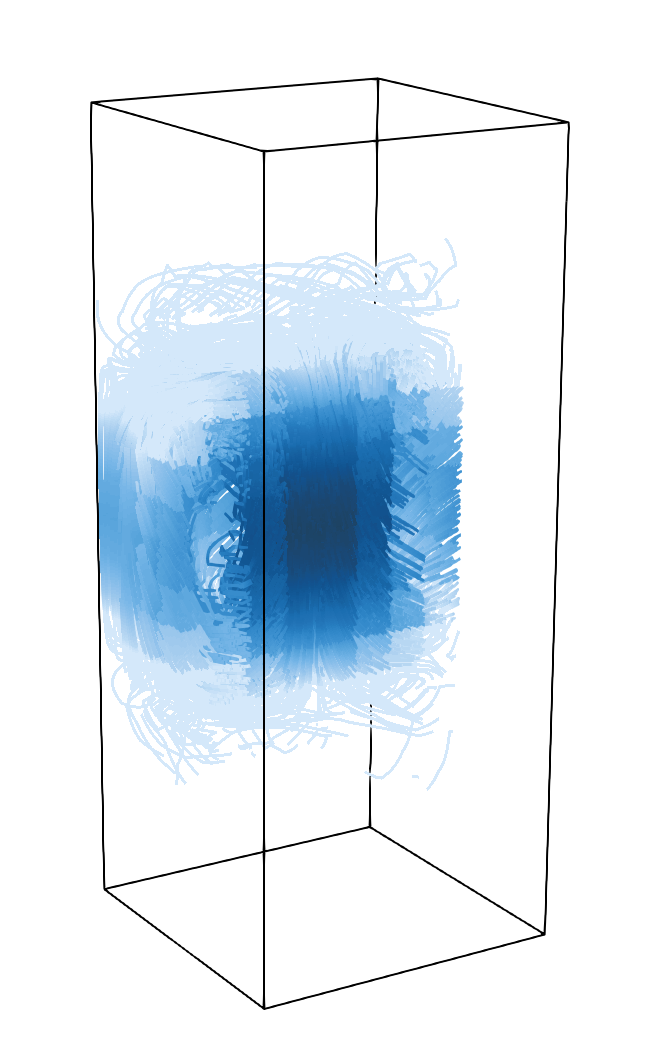}{0.33}{0} \\
    \hline
    \vspace{-4.5cm} Projection-based scheme \eqref{eq:mixed-scheme} with local constraints
    & \includegraphics[width=0.23\textwidth]{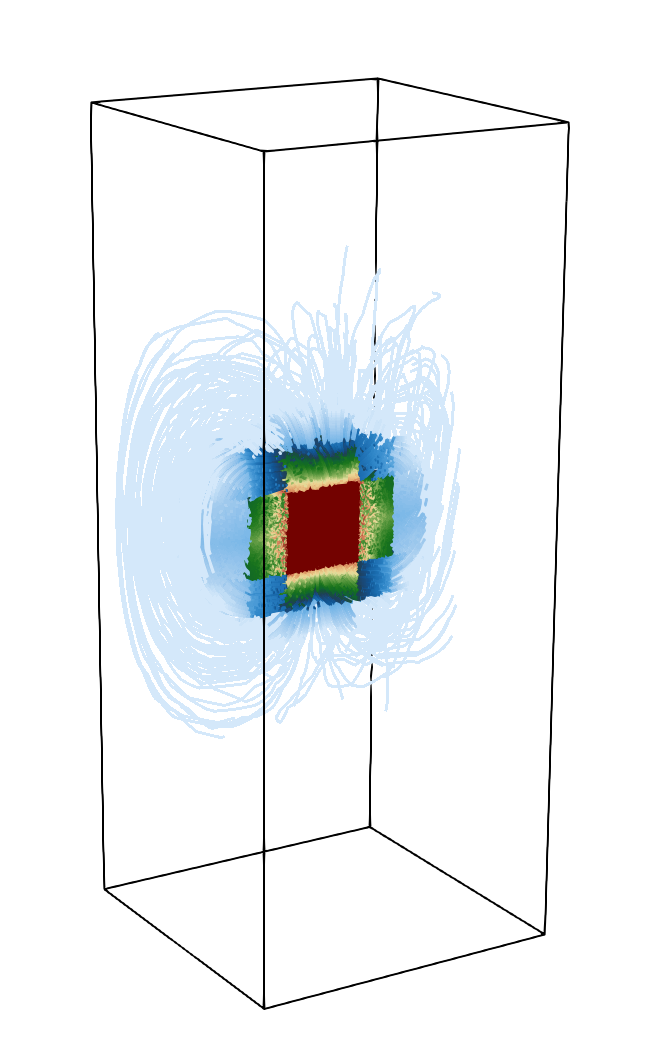}
    & \includegraphics[width=0.23\textwidth]{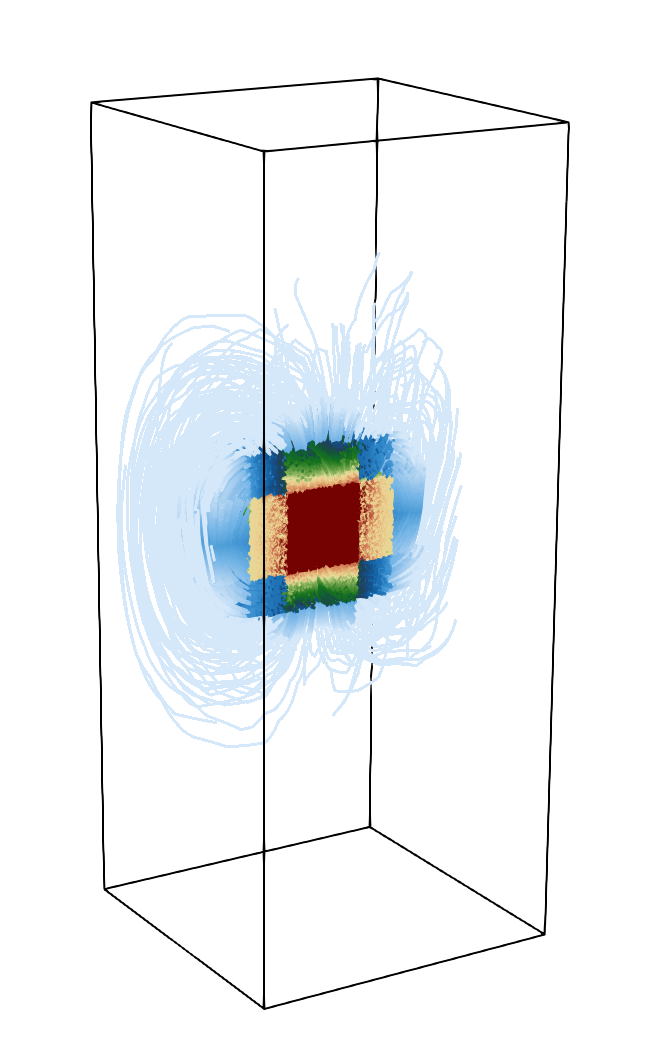}
    & \includegraphicswithlegend[0.85]{.26\textwidth}{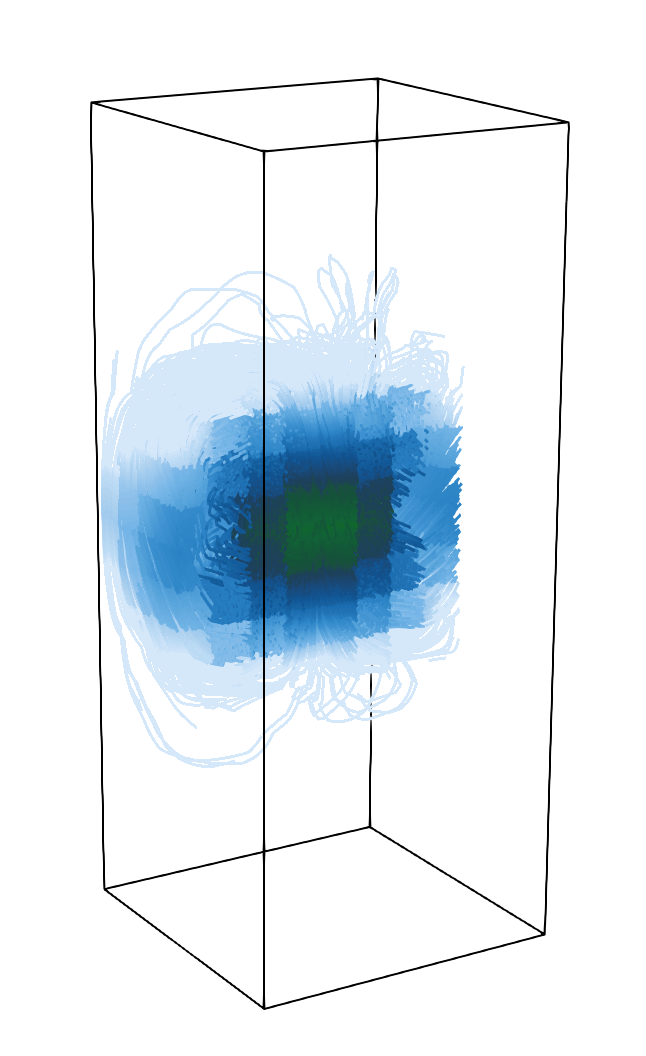}{0.33}{0}  \\ \hline
\end{tabular}
\caption{\emph{Magnetic knots}: comparison of evolution of stream tubes of the magnetic field under different topological constraints, colored by magnetic field strength $\|\bm B_h\|$.}
\label{fig:hopf-simulation}
\end{figure}

Based on these numerical results, several conclusions can be drawn. In the absence of the discrete Arnold inequality, the non-conservative scheme permits 
the magnetic energy to decay to (or very near) zero, confirming that without explicit topological protection, the magnetic field relaxes to a trivial state. 
Both the Lagrange multiplier and projection-based schemes, by contrast, converge 
to steady states with distinctly non-zero energy levels, consistent with Arnold's 
inequality, which asserts that nontrivial topology imposes a lower bound on the 
minimizable magnetic energy. The non-conservative 
scheme preserves only the discrete Gauss law, whereas both the Lagrange multiplier 
and projection-based schemes maintain the discrete Gauss law {\em and} global 
magnetic helicity to within machine precision and solver tolerances, thereby 
preserving the topological integrity of the field throughout the evolution.


\section{Local vs.~global helicity preservation: physical implications}\label{sec:discussion}
We have examined two helicity-preserving schemes for the magneto-frictional equations: one based on projection with auxiliary variables and the other based on a Lagrange multiplier approach. 
These schemes enforce different levels of helicity structure: the Lagrange
multiplier method imposes a single global constraint, whereas the projection-based
method preserves discrete local helicity on magnetically closed
subdomains. The numerical results show that this distinction is important, but
its effect depends strongly on the topology of the initial field.

This distinction is also relevant for interpreting the physical meaning of the two schemes. In ideal MHD, the magnetic flux through every material surface is conserved (Alfv\'en's flux theorem \cite{alfven1943existence}); equivalently, magnetic field lines are frozen into the flow. This is a much stronger, essentially local, constraint than the conservation of a single global helicity. It implies that field-line connectivity and the distribution of magnetic topology are transported by the flow, so that local helicities are preserved whenever the corresponding magnetically closed subdomains are advected with the plasma. 
In this sense, the projection-based scheme is closer to the ideal frozen-in structure of the magneto-frictional equations.

However, in reality, ideal MHD and its model problem \eqref{eqn:magneto-frictional} are only approximations of genuine plasma-physics systems, where diffusion of varying strengths and magnetic reconnection events occur during dynamical evolution. When diffusion and \emph{local} reconnection are present, local helicities are no longer conserved, whereas global magnetic helicity remains conserved (at least approximately, and often to high accuracy) \cite{taylor1974relaxation,pontin2016braided}.
A well-known example is Taylor relaxation, in which  small-scale reconnection events are believed to progressively rearrange magnetic topology while preserving the total helicity, thereby driving the system toward a minimum-energy state consistent with this single global invariant. This physical picture, originally proposed by Taylor, has proven remarkably successful: theoretical predictions closely match experimental observations in reverse-field pinches \cite{taylor1974relaxation}.


Solving the ideal MHD or the magneto-frictional equations with the ``wrong'' (i.e., merely global-helicity-conserving) scheme may, in fact, be closer to the real physics than the idealized equations themselves suggest. Numerical errors (relative to the true solutions of the ideal MHD or magneto-friction equations) arise from the weaker enforcement of the helicity constraint at local scales. Rather than viewing these deviations purely as undesirable artifacts, one can interpret them as mimicking physical local reconnection processes while still preserving helicity in a global, averaged sense.

This naturally raises the question: Can schemes based on the Lagrange multiplier approach be used to study Taylor relaxation and related phenomena? More broadly, can controlled numerical violations at small scales serve as a meaningful surrogate for unresolved physical reconnection processes, while preserving the correct global invariants? In short, the key issue is \emph{whether solving the ideal MHD equations with such ``incorrect'' algorithms can, in practice, yield physically more realistic solutions}.

In this work we do not provide a definitive answer. Instead, we view these results as opening a promising direction for further investigation. We also note that Faraco et al.~\cite{faraco2024magnetic} recently proposed a helicity conservation condition designed to better match the real physics of Taylor relaxation. Inspired by this and related work, a deeper understanding of the interplay between numerical constraint enforcement, topological evolution, and physical reconnection processes may yield new insights into relaxation theory and guide the design of structure-preserving numerical schemes for magnetohydrodynamics.

\section{Conclusion}\label{sec:conclusion}
In this work, we have investigated magnetic relaxation under three distinct
levels of helicity constraint: unconstrained evolution, global helicity
conservation, and local helicity conservation. By systematically comparing three finite element discretizations that enforce these constraints at the discrete
level, we have shown that the degree to which helicity is preserved numerically
can play an important role in determining the character of the relaxed magnetic
state.

The numerical results reveal three regimes. First, for helicity-carrying closed
fields, such as magnetic knots, helicity preservation imposes a nontrivial
constraint on the relaxation and changes the attainable steady state. Second,
for the WS braid, whose helicity vanishes, the corresponding
helicity constraint does not provide a positive lower bound on the magnetic
energy. In this case, the Lagrange multiplier and projection-based schemes may
relax toward qualitatively similar large-time states, although their transient
pathways and relaxation rates can differ. This illustrates a possible limitation of
helicity-based invariants for zero-helicity braids: nontrivial field-line
topology may remain invisible to the conserved helicity.

Third, for the CPH braid, the generalized helicity is nonzero and provides a
meaningful measure of the braided topology. In this case, both the Lagrange
multiplier and the projection-based schemes preserve this invariant and relax to
nonzero steady states. Thus, the main conclusion is more nuanced: helicity and
generalized helicity can provide effective topological barriers when they are
nonzero, but they do not form a complete description of magnetic topology. In
particular, zero-helicity braids require finer invariants or field-line-based
diagnostics. The decisive issue is therefore not only the level at which helicity
is preserved, but also whether the preserved helicity-type invariant is
informative for the topology of the field under consideration.

This observation also motivates the exploration of higher-order topological
invariants, such as those conjectured in
\cite{ArnoldTopologicalMethodsHydrodynamics2021}, as well as more refined
diagnostics such as field-line helicity distributions, field-line mappings and
topological entropy. Theoretical guarantees for the evolution of broad classes
of topologically nontrivial fields are available at the continuous level
\cite{freedman1988note}, but incorporating such information into practical
Eulerian finite element discretizations remains largely open.

This also raises an intriguing question: might the Lagrange multiplier approach better capture the essential physics of Taylor relaxation? This in turn raises a question on the best use of structure preservation in numerical PDEs: in real physical systems, which invariants should be rigidly preserved, and which may, or even should, be allowed to evolve or break on discretization?

Looking ahead, we plan to pursue several directions. Algorithmically, we intend
to develop efficient solvers for all three methods, including the application of
decoupling techniques and tailored preconditioners. Another interesting direction
is to study the interaction between spatial resolution, polynomial degree, and
structure preservation. One motivation for structure-preserving discretizations is to capture the
correct physics by enforcing the relevant invariants at the discrete level. At
the same time, structure preservation is also an efficiency question: a method
that builds in the correct topological constraints may reproduce physically
meaningful behaviour on much coarser meshes than a method that relies only on
resolution. In particular, it would be useful to
quantify how fine a mesh, or how high a polynomial degree, is needed for a
non-structure-preserving method to reproduce, to a prescribed tolerance, the
behaviour of a helicity-preserving discretization. Such a comparison would help
distinguish the effect of increasing approximation accuracy from the effect of
enforcing the relevant conservation law exactly at the discrete level. From the
physics perspective, we aim to apply these numerical schemes to study magnetic
braids in the context of solar coronal heating
\cite{pontin2016braided,Candelaresi2015}.

\section*{Code availability}
The simulations in \Cref{sec:Numerical-experiment} were implemented in {Firedrake} \cite{FiredrakeUserManual} and {PETSc} \cite{petsc-user-ref};
MUMPS~\cite{amestoy2001} was used to solve the linear systems.
The code used to generate the numerical results and all {Firedrake} components have been archived on Zenodo \cite{glrelax}
\section*{Acknowledgments}
We would like to thank Simon Candelaresi, Ralf Hiptmair, Gunnar Hornig,
Buyang Li, Shipeng Mao, Anthony Yeates and Enrico Zampa for helpful discussions. 

\appendix
\renewcommand*{\appendixname}{}

\section{Woltjer's variational principle}\label{app:vp}
Following Woltjer \cite{woltjer1958theorem}, we solve the following variational problem subject to the global helicity $\mathcal{H}$, viewing both quantities as functions of $\bm A$,
\begin{equation}
    \frac{\delta}{\delta \bm A} \left(\mathcal{E} - \alpha_0 \mathcal{H}\right) = 0.
\end{equation}
where $\alpha_0$ is a Lagrange multiplier to enforce the global helicity conservation $\mathcal{H}$.
Direct computation yields that
\begin{align*}
    \frac{\delta \mathcal{E}}{\delta \bm A} &= 2\int_{\Omega} \bm B\cdot \nabla\times \delta \bm A\  \d x,\\
    &=2\int_{\Omega} (\delta \bm A\cdot \nabla\times \bm B - \nabla\cdot [\bm B\times \delta \bm A])\ \d x,\\
    &=2\int_{\Omega} \delta \bm A \cdot \bm j\ \d x - 2\int_{\partial \Omega}\bm B\cdot (\delta \bm A\times \bm n)\ \d x, \\
    &=2\int_{\Omega} \delta \bm A\cdot\bm j\ \d x.
\end{align*}
Then, we compute
\begin{align*}
\frac{\delta \mathcal{H}}{\delta \bm A}
&= \int_\Omega \delta \bm{A} \cdot\bm{B}\ \d x 
  + \int_\Omega \bm{A} \cdot \delta \bm{B}\ \d x \\
&= \int_\Omega \delta \bm{A}\cdot\bm{B}\ \d x
  + \int_\Omega \bm{A}\cdot (\nabla\times \delta \bm{A})\ \d x \\
&= \int_\Omega \delta \bm{A}\cdot\bm{B}\ \d x
  + \int_\Omega \delta \bm{A}\cdot (\nabla\times \bm{A})\ \d x\\
  &\qquad - \int_{\partial \Omega} \bm{n}\cdot (\bm{A}\times \delta \bm{A})\ \d x \\
&= 2 \int_\Omega \delta \bm{A}\cdot\bm{B}\ \d x
   - \int_{\partial \Omega} \bm{A}\cdot (\delta \bm{A}\times \bm{n})\ \d x \\
&= 2 \int_\Omega \delta \bm{A}\cdot\bm{B}\ \d x .
\end{align*}
The minimum-energy state must satisfy
\begin{equation}\label{appB:mininal-state}
    \int_\Omega \delta \bm A \cdot(\bm j - \alpha_0 \bm B) \ \d x= 0.
\end{equation}
Therefore, we have $\bm j = \alpha_0 \bm B$, where $\alpha_0$ is a constant. 


\bibliographystyle{siam}
\bibliography{ref}

@article{finn1985magnetic,
  title={Magnetic helicity: what is it and what is it good for?},
  author={Finn, John M and Antonsen Jr, Thomas M},
  journal={Comments on Plasma Physics and Controlled Fusion},
  volume={9},
  number={3},
  pages={111--126},
  year={1985}
}

@article{berger1984topological,
  title   = {The topological properties of magnetic helicity},
  author  = {Berger, Mitchell A. and Field, George B.},
  journal = {Journal of Fluid Mechanics},
  volume  = {147},
  pages   = {133--148},
  year    = {1984},
  publisher = {Cambridge University Press},
  doi     = {10.1017/S0022112084002019}
}

@misc{glrelax,
  key    = {zenodo/Zenodo-20260708.2},
  author = {P. E. Farrell and M. He and K. Hu and G. Zhang},
  title  = {Software used in `Global and local helicity-preservation in the finite element discretisation of magnetic relaxation'},
  year   = {2026},
  doi    = {10.5281/zenodo.21263175},
  url    = {https://doi.org/10.5281/zenodo.21263175},
  note   = {\href{https://doi.org/10.5281/zenodo.21263175}
            {10.5281/zenodo.21263175}}
}

@article{faraco2024magnetic,
  title={Magnetic helicity, weak solutions and relaxation of ideal {MHD}},
  author={Faraco, Daniel and Lindberg, Sauli and Sz{\'e}kelyhidi Jr, L{\'a}szl{\'o}},
  journal={Communications on Pure and Applied Mathematics},
  volume={77},
  number={4},
  pages={2387--2412},
  year={2024},
  publisher={Wiley Online Library}
}

@article{freedman1988note,
  title={A note on topology and magnetic energy in incompressible perfectly conducting fluids},
  author={Freedman, Michael H},
  journal={Journal of Fluid Mechanics},
  volume={194},
  pages={549--551},
  year={1988},
  publisher={Cambridge University Press}
}

@article{da2025error,
  title={Error estimates for a helicity-preserving finite element discretisation of an incompressible magnetohydrodynamics system},
  author={Da Veiga, Lourenco Beirao and Hu, Kaibo and Mascotto, Lorenzo},
  journal={ESAIM: Mathematical Modelling and Numerical Analysis},
  volume={59},
  number={2},
  pages={1075--1094},
  year={2025},
  publisher={EDP Sciences}
}

@article{ma2016robust,
  title={{Robust preconditioners for incompressible MHD models}},
  author={Ma, Yicong and Hu, Kaibo and Hu, Xiaozhe and Xu, Jinchao},
  journal={Journal of Computational Physics},
  volume={316},
  pages={721--746},
  year={2016},
  publisher={Elsevier}
}

@article{massey1998higher,
  title={Higher order linking numbers},
  author={Massey, William S},
  journal={Journal of Knot Theory and Its Ramifications},
  volume={7},
  pages={393--414},
  year={1998},
  publisher={World Scientific}
}

@incollection{yeates2019magnetohydrodynamic,
  title={Magnetohydrodynamic relaxation theory},
  author={Yeates, Anthony R},
  booktitle={Topics in Magnetohydrodynamic Topology, Reconnection and Stability Theory},
  pages={117--143},
  year={2019},
  publisher={Springer}
}

@article{moffatt2015magnetic,
  title={Magnetic relaxation and the {Taylor} conjecture},
  author={Moffatt, HK},
  journal={Journal of Plasma Physics},
  volume={81},
  number={6},
  pages={905810608},
  year={2015},
  publisher={Cambridge University Press}
}

@article{tonnon2024semi,
  title={Semi-{L}agrangian finite element exterior calculus for incompressible flows},
  author={Tonnon, Wouter and Hiptmair, Ralf},
  journal={Advances in Computational Mathematics},
  volume={50},
  number={1},
  pages={11},
  year={2024},
  publisher={Springer}
}

@article{saad1993flexible,
  title={A flexible inner-outer preconditioned {GMRES} algorithm},
  author={Saad, Youcef},
  journal={SIAM Journal on Scientific Computing},
  volume={14},
  number={2},
  pages={461--469},
  year={1993},
  publisher={SIAM}
}

@article{YeatesHornig2013,
  author = {Yeates, A. R. and Hornig, G.},
  title = {Unique topological characterization of braided magnetic fields},
  journal = {Physics of Plasmas},
  volume = {20},
  number = {1},
  pages = {012102},
  year = {2013},
  doi = {10.1063/1.4817535}
}

@article{YeatesWilmotSmithHornig2010,
  author = {Yeates, A. R. and Wilmot-Smith, A. L. and Hornig, G.},
  title = {Topological Constraints on Magnetic Relaxation},
  journal = {Physical Review Letters},
  volume = {105},
  number = {8},
  pages = {085002},
  year = {2010},
  doi = {10.1103/PhysRevLett.105.085002}
}

@article{pontin2016braided,
  title={Braided magnetic fields: equilibria, relaxation and heating},
  author={Pontin, DI and Candelaresi, Simon and Russell, AJB and Hornig, Gunnar},
  journal={Plasma Physics and Controlled Fusion},
  volume={58},
  number={5},
  pages={054008},
  year={2016},
  publisher={IOP Publishing}
}

@article{Candelaresi2015,
  author    = {Candelaresi, S. and Pontin, D. I. and Hornig, G.},
  journal   = {The Astrophysical Journal},
  title     = {MAGNETIC FIELD RELAXATION AND CURRENT SHEETS IN AN IDEAL PLASMA},
  year      = {2015},
  issn      = {1538-4357},
  month     = jul,
  number    = {2},
  pages     = {134},
  volume    = {808},
  doi       = {10.1088/0004-637x/808/2/134},
  publisher = {American Astronomical Society},
}

@article{cheng2020new,
  title={A new {L}agrange multiplier approach for gradient flows},
  author={Cheng, Qing and Liu, Chun and Shen, Jie},
  journal={Computer Methods in Applied Mechanics and Engineering},
  volume={367},
  pages={113070},
  year={2020},
  publisher={Elsevier}
}

@article{cheng2020global,
  title={Global constraints preserving scalar auxiliary variable schemes for gradient flows},
  author={Cheng, Qing and Shen, Jie},
  journal={SIAM Journal on Scientific Computing},
  volume={42},
  number={4},
  pages={A2489--A2513},
  year={2020},
  publisher={SIAM}
}

@article{guo2023mass,
  title={Mass-, energy-, and momentum-preserving spectral scheme for {Klein}--{Gordon}--{Schr{\"o}dinger} system on infinite domains},
  author={Guo, Shimin and Mei, Liquan and Yan, Wenjing and Li, Ying},
  journal={SIAM Journal on Scientific Computing},
  volume={45},
  number={2},
  pages={B200--B230},
  year={2023},
  publisher={SIAM}
}

@article{garcke2025structure_twophase,
  title={Structure-preserving parametric finite element methods for two-phase {S}tokes flow based on {L}agrange multiplier approaches},
  author={Garcke, Harald and Trautwein, Dennis and Zhang, Ganghui},
  journal={arXiv preprint arXiv:2508.12326},
  year={2025}
}

@article{garcke2025structure,
  title={Structure-Preserving Parametric Finite Element Method for Surface Diffusion Based on {L}agrange Multiplier Approaches},
  author={Garcke, Harald and Jiang, Wei and Su, Chunmei and Zhang, Ganghui},
  journal={SIAM Journal on Scientific Computing},
  volume={47},
  number={3},
  pages={A1983--A2011},
  year={2025},
  publisher={SIAM}
}

@article{zhangMassKineticEnergy2022,
  title = {A Mass-, Kinetic Energy- and Helicity-Conserving Mimetic Dual-Field Discretization for Three-Dimensional Incompressible {{Navier-Stokes}} Equations, Part {{I}}: {{Periodic}} Domains},
  shorttitle = {A Mass-, Kinetic Energy- and Helicity-Conserving Mimetic Dual-Field Discretization for Three-Dimensional Incompressible {{Navier-Stokes}} Equations, Part {{I}}},
  author = {Zhang, Yi and Palha, Artur and Gerritsma, Marc and Rebholz, Leo G.},
  year = {2022},
  journal = {Journal of Computational Physics},
  volume = {451},
  pages = {110868},
  issn = {00219991},
  doi = {10.1016/j.jcp.2021.110868},
  urldate = {2025-06-08},
  langid = {english}
}

@article{zhangMEEVCDiscretizationTwodimensional2024,
  title = {A {{MEEVC}} Discretization for Two-Dimensional Incompressible {{Navier-Stokes}} Equations with General Boundary Conditions},
  author = {Zhang, Yi and Palha, Artur and Gerritsma, Marc and Yao, Qinghe},
  year = {2024},
  journal = {Journal of Computational Physics},
  volume = {510},
  pages = {113080},
  publisher = {Elsevier BV},
  issn = {0021-9991},
  doi = {10.1016/j.jcp.2024.113080},
  urldate = {2025-07-17},
  copyright = {https://www.elsevier.com/tdm/userlicense/1.0/},
  langid = {english}
}

@article{maoIncompressibilityDivB0Preserving2025,
  title = {An Incompressibility, {{divB}}=0 Preserving, Current Density, Helicity, Energy-Conserving Finite Element Method for Incompressible {{MHD}} Systems},
  author = {Mao, Shipeng and Xi, Ruijie},
  year = {2025},
  journal = {Journal of Computational Physics},
  volume = {538},
  pages = {114130},
  publisher = {Elsevier BV},
  issn = {0021-9991},
  doi = {10.1016/j.jcp.2025.114130},
  urldate = {2025-07-10},
  copyright = {https://www.elsevier.com/tdm/userlicense/1.0/},
  langid = {english}
}

@article{BlickhanMRXdifferentiable3D2025,
  title = {{{MRX}}: {{A}} Differentiable {{3D MHD}} Equilibrium Solver without Nested Flux Surfaces},
  shorttitle = {{{MRX}}},
  author = {Blickhan, Tobias and Stratton, Julianne and Kaptanoglu, Alan A.},
  year = {2025},
  number = {arXiv:2510.26986},
  eprint = {2510.26986},
  primaryclass = {physics},
  publisher = {arXiv},
  doi = {10.48550/arXiv.2510.26986},
  urldate = {2025-11-14},
  archiveprefix = {arXiv},
  langid = {english},
  journal = {arXiv preprint arXiv:2510.26986},
}

@article{he2025helicity,
  title={Helicity-preserving finite element discretization for magnetic relaxation},
  author={He, M and Farrell, PE and Hu, K and Andrews, B},
  journal={SIAM Journal on Scientific Computing},
  year={2025},
doi = {10.1137/25M1727540},
volume = {48},
issue = {2},
pages = {B165--B183},
}

@article{alfven1943existence,
  title   = {On the existence of electromagnetic-hydrodynamic waves},
  author  = {Alfv{\'e}n, Hannes},
  journal = {Arkiv for matematik, astronomi och fysik},
  volume  = {29},
  pages   = {1--7},
  year    = {1943}
}

@article{amestoy2001,
  author  = {P. R. Amestoy and I. S. Duff and J. Koster and J.-Y. L'Excellent},
  title   = {A fully asynchronous multifrontal solver using distributed dynamic scheduling},
  journal = {SIAM Journal on Matrix Analysis and Applications},
  volume  = {23},
  number  = {1},
  year    = {2001},
  pages   = {15--41}
}

@article{andrews2024enforcing,
  title   = {Enforcing conservation laws and dissipation inequalities numerically via auxiliary variables},
  author  = {Andrews, Boris D and Farrell, Patrick E},
  year    = {2025},
  Journal = {SIAM Journal on Scientific Computing},
  doi = {10.1137/25M1756673},
  volume = {47},
  issue = {6},
}

@article{arnold1974asymptotic,
  title     = {The asymptotic {H}opf invariant and its applications},
  author    = {Arnold, Vladimir I},
  journal   = {Vladimir I. Arnold-Collected Works: Hydrodynamics, Bifurcation Theory, and Algebraic Geometry 1965-1972},
  pages     = {357--375},
  year      = {1974},
  publisher = {Springer}
}

@article{arnoldFiniteElementExterior2006,
  title   = {Finite Element Exterior Calculus, Homological Techniques, and Applications},
  author  = {Arnold, Douglas N. and Falk, Richard S. and Winther, Ragnar},
  year    = {2006},
  journal = {Acta Numerica},
  volume  = {15},
  pages   = {1--155},
  issn    = {0962-4929, 1474-0508},
  doi     = {10.1017/S0962492906210018},
  urldate = {2023-11-09},
  langid  = {english}
}

@article{arnoldFiniteElementExterior2010,
  title      = {Finite Element Exterior Calculus: from {{Hodge}} Theory to Numerical Stability},
  shorttitle = {Finite Element Exterior Calculus},
  author     = {Arnold, Douglas and Falk, Richard and Winther, Ragnar},
  year       = {2010},
  journal    = {Bulletin of the American Mathematical Society},
  volume     = {47},
  number     = {2},
  pages      = {281--354},
  issn       = {0273-0979, 1088-9485},
  doi        = {10.1090/S0273-0979-10-01278-4},
  urldate    = {2023-11-09},
  langid     = {english}
}

@book{ArnoldFiniteElementExterior2018,
  title     = {Finite {{Element Exterior Calculus}}},
  author    = {Arnold, Douglas N.},
  year      = {2018},
  publisher = {{Society for Industrial and Applied Mathematics}},
  address   = {Philadelphia, PA},
  doi       = {10.1137/1.9781611975543},
  urldate   = {2024-01-06},
  isbn      = {978-1-61197-553-6 978-1-61197-554-3},
  langid    = {english}
}

@book{ArnoldTopologicalMethodsHydrodynamics2021,
  title     = {Topological {{Methods}} in {{Hydrodynamics}}},
  author    = {Arnold, Vladimir I. and Khesin, Boris A.},
  year      = {2021},
  series    = {Applied {{Mathematical Sciences}}},
  volume    = {125},
  publisher = {Springer International Publishing},
  address   = {Cham},
  doi       = {10.1007/978-3-030-74278-2},
  urldate   = {2024-01-06},
  isbn      = {978-3-030-74277-5 978-3-030-74278-2},
  langid    = {english}
}

@techreport{bevir1980relaxation,
  title       = {Relaxation, flux consumption and quasi steady state pinches},
  author      = {Bevir, MK and Gray, JW},
  year        = {1980},
  institution = {Los Alamos National Laboratory},
  number      = {LA-8944-C}
}

@article{brezzi1985two,
  title     = {Two families of mixed finite elements for second order elliptic problems},
  author    = {Brezzi, Franco and Douglas, Jim and Marini, L. Donatella},
  journal   = {Numerische Mathematik},
  volume    = {47},
  pages     = {217--235},
  year      = {1985},
  publisher = {Springer}
}

@article{candelaresiMimeticMethodsLagrangian2014a,
 title={Mimetic methods for {L}agrangian relaxation of magnetic fields},
  author={Candelaresi, Simon and Pontin, David and Hornig, Gunnar},
  journal={SIAM Journal on Scientific Computing},
  volume={36},
  number={6},
  pages={B952--B968},
  year={2014},
  publisher={SIAM}
}

@article{Chodura3DcodeMHD1981,
  title     = {A {{3D}} Code for {{MHD}} Equilibrium and Stability},
  author    = {Chodura, R. and Schl{\"u}ter, A.},
  year      = {1981},
  journal   = {Journal of Computational Physics},
  volume    = {41},
  number    = {1},
  pages     = {68--88},
  issn      = {00219991},
  doi       = {10.1016/0021-9991(81)90080-2},
  urldate   = {2024-08-29},
  copyright = {https://www.elsevier.com/tdm/userlicense/1.0/},
  langid    = {english}
}

@article{craig1986dynamic,
  title   = {A dynamic relaxation technique for determining the structure and stability of coronal magnetic fields},
  author  = {Craig, I. J. D. and Sneyd, A. D.},
  journal = {The Astrophysical Journal},
  volume  = {311},
  pages   = {451--459},
  year    = {1986}
}

@article{craig2005parker,
  title     = {The {Parker} problem and the theory of coronal heating},
  author    = {Craig, I. J. D. and Sneyd, A. D.},
  journal   = {Solar Physics},
  volume    = {232},
  pages     = {41--62},
  year      = {2005},
  publisher = {Springer}
}

@article{craig2014current,
  title     = {Current singularities in line-tied three-dimensional magnetic fields},
  author    = {Craig, I. J. D. and Pontin, DI},
  journal   = {The Astrophysical Journal},
  volume    = {788},
  number    = {2},
  pages     = {177},
  year      = {2014},
  publisher = {IOP Publishing}
}

@manual{FiredrakeUserManual,
  title        = {Firedrake User Manual},
  author       = {David A. Ham and Paul H. J. Kelly and Lawrence Mitchell and Colin J. Cotter and Robert C. Kirby and Koki Sagiyama and Nacime Bouziani and Sophia Vorderwuelbecke and Thomas J. Gregory and Jack Betteridge and Daniel R. Shapero and Reuben W. Nixon-Hill and Connor J. Ward and Patrick E. Farrell and Pablo D. Brubeck and India Marsden and Thomas H. Gibson and Miklós Homolya and Tianjiao Sun and Andrew T. T. McRae and Fabio Luporini and Alastair Gregory and Michael Lange and Simon W. Funke and Florian Rathgeber and Gheorghe-Teodor Bercea and Graham R. Markall},
  organization = {Imperial College London and University of Oxford and Baylor University and University of Washington},
  edition      = {First edition},
  year         = {2023},
  month        = {5},
  doi          = {10.25561/104839}
}

@article{gawlikFiniteElementMethod2022,
  title   = {A Finite Element Method for {{MHD}} That Preserves Energy, Cross-Helicity, Magnetic Helicity, Incompressibility, and Div {{B}} = 0},
  author  = {Gawlik, Evan S. and {Gay-Balmaz}, Fran{\c c}ois},
  year    = {2022},
  journal = {Journal of Computational Physics},
  volume  = {450},
  pages   = {110847},
  issn    = {00219991},
  doi     = {10.1016/j.jcp.2021.110847},
  urldate = {2024-01-25},
  langid  = {english}
}

@article{huHelicityconservativeFiniteElement2021,
  title   = {Helicity-Conservative Finite Element Discretization for Incompressible {{MHD}} Systems},
  author  = {Hu, Kaibo and Lee, Young-Ju and Xu, Jinchao},
  year    = {2021},
  journal = {Journal of Computational Physics},
  volume  = {436},
  pages   = {110284},
  issn    = {00219991},
  doi     = {10.1016/j.jcp.2021.110284},
  urldate = {2023-12-24},
  langid  = {english}
}

@article{LaakmannStructurepreservinghelicityconservingfinite2023,
  title   = {Structure-Preserving and Helicity-Conserving Finite Element Approximations and Preconditioning for the {{Hall MHD}} Equations},
  author  = {Laakmann, Fabian and Hu, Kaibo and Farrell, Patrick E.},
  year    = {2023},
  journal = {Journal of Computational Physics},
  volume  = {492},
  pages   = {112410},
  issn    = {00219991},
  doi     = {10.1016/j.jcp.2023.112410},
  urldate = {2024-01-07},
  langid  = {english}
}

@article{longbottom1998magnetic,
  title     = {Magnetic flux braiding: force-free equilibria and current sheets},
  author    = {Longbottom, AW and Rickard, GJ and Craig, I. J. D. and Sneyd, A. D.},
  journal   = {The Astrophysical Journal},
  volume    = {500},
  number    = {1},
  pages     = {471},
  year      = {1998},
  publisher = {IOP Publishing}
}

@article{hu2019structure,
  title={{Structure-preserving finite element methods for stationary MHD models}},
  author={Hu, Kaibo and Xu, Jinchao},
  journal={Mathematics of Computation},
  volume={88},
  number={316},
  pages={553--581},
  year={2019}
}

@article{hu2017stable,
  title={{Stable finite element methods preserving $\nabla\cdot \bm{B}=0$ exactly for MHD models}},
  author={Hu, Kaibo and Ma, Yicong and Xu, Jinchao},
  journal={Numerische Mathematik},
  volume={135},
  number={2},
  pages={371--396},
  year={2017},
  publisher={Springer}
}

@article{nedelec1-0,
  author  = {N\'ed\'elec, Jean-Claude},
  title   = {Mixed finite elements in \(\mathbb{R}^3\)},
  journal = {Numerische Mathematik},
  volume  = {35},
  number  = {3},
  year    = {1980},
  doi     = {10.1007/BF01396415},
  pages   = {{315--341}}
}

@techreport{petsc-user-ref,
  author      = {Satish Balay and Shrirang Abhyankar and Mark~F. Adams and Steven Benson and Jed
                 Brown and Peter Brune and Kris Buschelman and Emil Constantinescu and Lisandro
                 Dalcin and Alp Dener and Victor Eijkhout and Jacob Faibussowitsch and William~D.
                 Gropp and V\'{a}clav Hapla and Tobin Isaac and Pierre Jolivet and Dmitry Karpeev
                 and Dinesh Kaushik and Matthew~G. Knepley and Fande Kong and Scott Kruger and
                 Dave~A. May and Lois Curfman McInnes and Richard Tran Mills and Lawrence Mitchell
                 and Todd Munson and Jose~E. Roman and Karl Rupp and Patrick Sanan and Jason Sarich
                 and Barry~F. Smith and Hansol Suh and Stefano Zampini and Hong Zhang and Hong Zhang
                 and Junchao Zhang},
  title       = {{PETSc/TAO} Users Manual},
  institution = {Argonne National Laboratory},
  number      = {ANL-21/39 - Revision 3.22},
  doi         = {10.2172/2205494},
  year        = {2024}
}

@inproceedings{raviart2006mixed,
  title        = {A mixed finite element method for 2-nd order elliptic problems},
  author       = {Raviart, Pierre-Arnaud and Thomas, Jean-Marie},
  booktitle    = {Mathematical Aspects of Finite Element Methods: Proceedings of the Conference Held in Rome, December 10--12, 1975},
  pages        = {292--315},
  year         = {2006},
  organization = {Springer}
}

@article{smietIdealRelaxationHopf2017,
  title         = {Ideal {{relaxation}} of the {{Hopf fibration}}},
  author        = {Smiet, Christopher Berg and Candelaresi, Simon and Bouwmeester, Dirk},
  year          = {2017},
  journal       = {Physics of Plasmas},
  volume        = {24},
  number        = {7},
  eprint        = {1610.04719},
  primaryclass  = {physics},
  pages         = {072110},
  issn          = {1070-664X, 1089-7674},
  doi           = {10.1063/1.4990076},
  urldate       = {2024-03-30},
  archiveprefix = {arXiv},
  langid        = {english}
}

@article{taylor1974relaxation,
  title     = {Relaxation of toroidal plasma and generation of reverse magnetic fields},
  author    = {Taylor, J Brian},
  journal   = {Physical Review Letters},
  volume    = {33},
  number    = {19},
  pages     = {1139},
  year      = {1974},
  publisher = {APS}
}

@article{wilmot-smith2009,
  title   = {{{MAGNETIC BRAIDING AND PARALLEL ELECTRIC FIELDS}}},
  author  = {{Wilmot-Smith}, A. L. and Hornig, G. and Pontin, D. I.},
  year    = {2009},
  journal = {The Astrophysical Journal},
  volume  = {696},
  number  = {2},
  pages   = {1339--1347},
  issn    = {0004-637X, 1538-4357},
  doi     = {10.1088/0004-637X/696/2/1339},
  urldate = {2024-05-30},
  langid  = {english}
}

@article{wilmot2009magnetic,
  title={Magnetic braiding and quasi-separatrix layers},
  author={Wilmot-Smith, AL and Hornig, G and Pontin, DI},
  journal={The Astrophysical Journal},
  volume={704},
  number={2},
  pages={1288--1295},
  year={2009},
  publisher={The American Astronomical Society}
}

@article{wilmot2009magneticparallel,
  title     = {Magnetic braiding and parallel electric fields},
  author    = {Wilmot-Smith, AL and Hornig, G and Pontin, DI},
  journal   = {The Astrophysical Journal},
  volume    = {696},
  number    = {2},
  pages     = {1339},
  year      = {2009},
  publisher = {IOP Publishing}
}

@article{woltjer1958theorem,
  title     = {A theorem on force-free magnetic fields},
  author    = {Woltjer, Lodewijk},
  journal   = {Proceedings of the National Academy of Sciences},
  volume    = {44},
  number    = {6},
  pages     = {489--491},
  year      = {1958},
  publisher = {National Acad Sciences}
}

@article{yeatesLimitationsMagnetofrictionalRelaxation2022,
  title   = {On the Limitations of Magneto-Frictional Relaxation},
  author  = {Yeates, Anthony R.},
  year    = {2022},
  journal = {Geophysical \& Astrophysical Fluid Dynamics},
  volume  = {116},
  number  = {4},
  pages   = {305--320},
  issn    = {0309-1929, 1029-0419},
  doi     = {10.1080/03091929.2021.2021197},
  urldate = {2024-07-29},
  langid  = {english}
}

@article{zhou2014variational,
  title     = {Variational integration for ideal magnetohydrodynamics with built-in advection equations},
  author    = {Zhou, Yao and Qin, Hong and Burby, Joshua W and Bhattacharjee, Amitava},
  journal   = {Physics of Plasmas},
  volume    = {21},
  number    = {10},
  year      = {2014},
  publisher = {AIP Publishing}
}

@article{zhou2016formation,
  title     = {Formation of current singularity in a topologically constrained plasma},
  author    = {Zhou, Yao and Huang, Yi-Min and Qin, Hong and Bhattacharjee, Amitava},
  journal   = {Physical Review E},
  volume    = {93},
  number    = {2},
  pages     = {023205},
  year      = {2016},
  publisher = {APS}
}

@article{zhou2017constructing,
  title     = {Constructing current singularity in a {3D} line-tied plasma},
  author    = {Zhou, Yao and Huang, Yi-Min and Qin, Hong and Bhattacharjee, Amitava},
  journal   = {The Astrophysical Journal},
  volume    = {852},
  number    = {1},
  pages     = {3},
  year      = {2017},
  publisher = {IOP Publishing}
}

@article{murphy2000,
author = {M. F. Murphy and G. H. Golub and A. J. Wathen},
title = {A note on preconditioning for indefinite linear systems},
journal = {SIAM Journal on Scientific Computing},
volume = {21},
number = {6},
pages = {1969-1972},
year = {2000},
doi = {10.1137/S1064827599355153},
}

@article{ipsen2001,
author = {I. C. F. Ipsen},
title = {A note on preconditioning nonsymmetric matrices},
journal = {SIAM Journal on Scientific Computing},
volume = {23},
number = {3},
pages = {1050--1051},
year = {2001},
doi = {10.1137/S1064827500377435},
}

\end{document}